\documentclass[a4paper]{article}
\usepackage[latin1]{inputenc}
\usepackage[T1]{fontenc}
\usepackage{amsmath}
\usepackage{amsfonts}
\usepackage{amssymb}
\usepackage{graphicx}
\usepackage{hyperref}
\usepackage{subfigure}
\usepackage{amsthm}  

\newtheorem{Theorem}{Theorem}
\newtheorem{Lemma}{Lemma}
\newtheorem{Example}{Example}
\newtheorem{Remark}{Remark}
\newtheorem{Corollary}{Corollary}
\usepackage[figurename=Fig.]{caption}

\usepackage{geometry}
\begin{document}

\centerline{\bf \Large Spherically Restricted Random Hyperbolic Diffusion}

\

{Philip Broadbridge $^{1}$ Alexander D. Kolesnik $^{2}$ Nikolai Leonenko $^{3}$ Andriy Olenko $^{1}$ Dareen Omari $^{1}$\\	
	
	\
	
	\noindent	$^{1}$ \quad Department of Mathematics and Statistics, La Trobe University, Melbourne, VIC, 3086, Australia; \\ P.Broadbridge@latrobe.edu.au; a.olenko@latrobe.edu.au;  omari.d@students.latrobe.edu.au
	\\
	\noindent	$^{2}$ \quad Institute of Mathematics and Computer Science, Academy Street 5, Kishinev 2028, Moldova;\\ kolesnik@math.md
	\\
	\noindent	$^{3}$ \quad School of Mathematics, Cardiff University, Senghennydd Road, Cardiff, Wales, UK, CF24 4AG; \\ leonenkon@cardiff.ac.uk.
}	

	\begin{abstract}{This paper investigates solutions of hyperbolic diffusion equations in $\mathbb{R}^3$ with random initial conditions. The solutions are given as spatial-temporal random fields. Their restrictions to the unit sphere $S^2$ are studied.  All assumptions are formulated in terms of the angular power spectrum or the spectral measure of the random initial conditions. Approximations to the exact solutions are given. Upper bounds for the mean-square convergence rates of the approximation fields are obtained. The smoothness properties of the exact solution and its approximation are also investigated. It is demonstrated that the H\"{o}lder-type continuity of the solution depends on the decay of the angular power spectrum. Conditions on the spectral measure of initial conditions that guarantee short or long-range dependence of the solutions are given.	
		Numerical studies are presented to verify the theoretical findings.}
\end{abstract}

	% Keywords
\noindent{Stochastic partial differential equations; Hyperbolic diffusion equation; Spherical random field; H\"{o}lder continuity; Long-range dependence; Approximation errors; Cosmic microwave background }
	
\section{Introduction}
Numerous environmental, biological and astrophysical applications require modelling of changes in data on the unit sphere $S^2$ or in the $3D$ space $\mathbb{R}^3$, see~\cite{ weinberg2008cosmology,broadbridge2006dark,marinucci2011random, ade2016planck,angulo2008spatiotemporal,stein2007spatial}. One of conventional tools for such modelling is stochastic partial differential equations (SPDEs), see, for example, \cite{stein2007spatial, angulo2008spatiotemporal, anh2018approximation, broadbridge2019random} and the references therein.
Random fields that are solutions of such SPDEs often exhibit dynamics dependent on initial conditions. Properties of these random fields are subjects of investigation, important both practically and theoretically.

Spherical random fields have been used as a standard model in the astrophysical and cosmological literature in the last decades, see \cite{marinucci2011random, ade2016planck, broadbridge2019random, hamann2019new}. NASA and ESA space missions \cite{ade2016planck} obtained very detailed measurements of Cosmic Microwave Background radiation (CMB), which are interpreted as a realisation of a spherical random field superimposed on an underlying signal of large-scale acoustic waves in plasma near the time of recombination. The theory of the standard inflation scenario uses a Gaussian model for the density fluctuation of this field, see \cite{weinberg2008cosmology,marinucci2011random,ade2016planck}. Several new cosmological models were proposed using non-Gaussian assumptions and employed sophisticated statistical tests to justify possible departures from Gaussianity. The understanding of changes in CMB temperature fluctuations is important to predict future cosmological evolution and accurately reconstruct past states of the universe.
It also can help in the estimation and statistical inference of physical parameters obtained from the CMB data.

SPDEs on $\mathbb{R}^3$ have been extensively studied. However, SPDEs on manifolds attracted a lot of attention only recently, see \cite{lang2015isotropic,anh2018approximation, broadbridge2019random, lan2019regularity}. The results in these papers demonstrate that the continuity properties of solutions and convergence rates of approximations to solutions are determined by decay rates of the angular power spectrum of initial random conditions. This article continues studies of solutions of SPDEs on the sphere. However, in contrast to the above publications that directly model spherical random fields using Laplace or Laplace-Beltrami operators on the sphere, we employ another approach. Namely, we consider the restriction of the stochastic hyperbolic diffusion in $\mathbb{R}^3$ to the unit sphere. Compared to the available literature this approach is more consistent with real CMB observations that exist in $3D$ space but are measured only on $S^2$. 
From a mathematical point of view, additional investigations are required to show that solutions of known models on the sphere admit physically meaningful extensions to $\mathbb{R}^3$ that are consistent with $3D$ observations.
By its construction, our model directly provides this consistency.
The proposed model may find new applications for the next generation of CMB experiments, CMB-$S^4$, which will be collecting $3D$ observations.
A very detailed discussion of SPDEs on manifolds and their physical and mathematical justification for CMB problems can be found in \cite{broadbridge2019random}. The hyperbolic diffusion equation prohibits superluminal propagation of density disturbances that is an unwanted feature of pure diffusion models over super-galactic distances.  In addition, the linear hyperbolic diffusion equation, expressed in terms of co-moving material space coordinates and conformal time coordinate, is a good approximation to the field equation of a scalar field minimally coupled to an expanding Robertson-Walker space-time. However, speed-limited diffusion raises some interesting questions about the dynamics of Shannon entropy. For physical concentrations governed by linear or nonlinear heat diffusion equations of parabolic type, Shannon entropy is fully analogous to thermodynamic entropy and it increases monotonically \cite{broadbridge2008entropy}. It will be explained that at low wave numbers, the hyperbolic diffusion equation behaves as a dissipative diffusion equation but above some cut-off wave number it behaves as a bi-directional wave equation which has increasing entropy when twin pulses separate but has decreasing entropy when pulses approach each other constructively.

The paper is organized as follows. 
Section~\ref{sec5.2} presents definitions and results about spatial-temporal random fields in $\mathbb{R}^3$. It also introduces hyperbolic diffusion equations with random initial conditions and their solutions. Section~\ref{sec5.3} investigates the spatial-temporal hyperbolic diffusion field on the unit sphere. The 
H\"{o}lder-type continuity of the exact solution of the spatial-temporal hyperbolic diffusion field on the sphere is investigated in Section
\ref{sec5.8}. 
In Section~\ref{sec5.4} we
study the dependence structures of the spherical hyperbolic diffusion random fields.
Section~\ref{sec5.5} obtains the mean-square convergence rate to the diffusion field in terms of the angular power spectrum. Section~\ref{sec5.6} provides some numerical results. Finally, Shannon entropy behaviour is discussed in Section~\ref{sec5.9}, followed by some conclusions. 

All numerical computations and simulations in this paper were performed using the software R version 3.6.1 and Python version 3.7.5. The results were derived using the HEALPix representation of spherical data, see \cite{gor} and \mbox{\href{http://healpix.sourceforge.net}{http://healpix.sourceforge.net}}. In particular, the R package rcosmo \cite{fryer2019rcosmo}, \cite{fryer2018rcosmo} was used for computations and visualisations of the obtained results. The Python package healpy  was used for fast spherical harmonics generation of spherical maps from Laplace series coefficients. The R and Python code used for numerical examples in Section~\ref{sec5.6} are freely available in the folder "Research materials" from the  website \mbox{\href{https://sites.google.com/site/olenkoandriy/}{https://sites.google.com/site/olenkoandriy/}}.

We will use the symbol C to denote constants that are not important for our exposition. The same symbol may be used for different constants appearing in the same proof.
\section{Spatial random hyperbolic diffusion}\label{sec5.2}
\numberwithin{equation}{section}
This section reviews the basic theory of random fields in $\mathbb{R}^3$ and introduces a hyperbolic diffusion with random initial conditions. Then the solution of the diffusion equation is derived and analysed.
We consider the hyperbolic diffusion equation
\begin{equation}\label{telegraph}
	\frac{1}{c^{2}}\frac{\partial ^{2} q(\mathbf{x},t)}{\partial t^{2}}+ 
	\frac{1}{D}\frac{\partial q(\mathbf{x},t)}{\partial t}= \Delta q(\mathbf{x},t),  
\end{equation}
$$\mathbf{x}=(x_{1},x_{2},x_{3})\in\Bbb R^3 , \quad t\geq 0, \quad D>0, \quad c>0,$$
subject to the random initial conditions:
\begin{equation}\label{incond}
	q(\mathbf{x},t)|_{t=0}=\eta (\mathbf{x}), \qquad \left. \frac{\partial q(\mathbf{x},t)}{\partial t}\right\vert _{t=0}=0,
\end{equation}
where $\Delta$ is the Laplacian in $\Bbb R^3$ and the random field $\eta(\bold x)=\eta(\bold x,\omega), \; \bold x\in\Bbb R^3, \; \omega\in\Omega$, defined on a suitable complete probability space $(\Omega,\mathcal F, P)$, is assumed to be a measurable, mean-square continuous, wide-sense homogeneous and isotropic with zero mean and the  covariance function $B(\Vert\bold x-\bold y\Vert)  = \text{Cov}(\eta(\bold x),\eta(\bold y))$.

The covariance function has the following representation
\begin{equation*}%\label{covfunc}
	\aligned 
	B(\Vert\bold x-\bold y\Vert) & 
	= \int_{\Bbb R^3} \cos(\langle  \kappa, \bold x-\bold y \rangle) \; F(d \kappa)
	= \int_0^{\infty} \frac{\sin(\mu \Vert\bold x-\bold y\Vert)}{\mu \Vert\bold x-\bold y\Vert} \; G(d\mu) , 
	\endaligned
\end{equation*}
for some bounded, non-negative measures $F(\cdot)$ on $(\Bbb R^3, \mathcal B(\Bbb R^3))$ and $G(\cdot)$ on 
$(\Bbb R_+^1, \mathcal B(\Bbb R_+^1))$, such that 
$$F(\Bbb R^3)=G([0,\infty)) = B(0), \qquad G(\mu) = \int\limits_{\{\Vert \kappa\Vert < \mu\}} F(d \kappa),$$ 
see \cite{yadrenko1983spectral}, pp. 1-5 and \cite{ivanov1989statistical}, pp. 10-15 for more details.

Then there exists a complex-valued orthogonally scattered random measure $Z(\cdot)$ 
such that, for every $\bold x\in\Bbb R^3$, the field $\eta(\bold x)$ itself has the spectral representation 
\begin{equation}\label{spectrepres}
	\eta(\bold x) = \int_{\Bbb R^3} e^{i\langle  \kappa, \bold x \rangle} \; Z(d \kappa) ,\quad \ \bold E |Z(\Delta)|^2 = F(\Delta), \qquad \Delta\in\mathcal B(\Bbb R^3).
\end{equation}
Let $Y_{lm}(\theta,\varphi)$, $\theta \in [0, \pi]$, $\varphi \in [0,2\pi)$, $l=0,1,\dots$, $m=-l,\dots,l$, be complex spherical harmonics defined by the relation 
$$Y_{lm}(\theta,\varphi)=(-1)^m \bigg(\frac{(2l+1)(l-m)!}{4\pi(l+m)!}  \bigg)^{1/2}\exp(im\varphi)P_l^{m}(\cos(\theta)), $$
where $P_l^{m}(\cdot)$ are the associated Legendre polynomials with indices $l$ and $m$.
For spherical harmonics it holds 
$$Y_{l0}(0,0)=\sqrt{\frac{2l+1}{4\pi}},\ \ \quad Y_{l0}(\theta,\varphi)=\sqrt{\frac{2l+1}{4\pi}}P_{l0}(\cos \theta),$$ 
$$Y_{lm}^{*}(\theta,\varphi)=(-1)^mY_{l(-m)}(\theta,\varphi),$$
$$Y_{lm}(\pi-\theta,\varphi+\pi)=(-1)^lY_{lm}(\theta,\varphi),$$  
$$\int_{0}^{\pi}\int_{0}^{2\pi}Y_{lm}^{*}(\theta,\varphi)Y_{l'm'}(\theta,\varphi)\sin \theta d\varphi d\theta=\delta_l^{l'}\delta_m^{m'},$$ 
where the symbol * denotes the complex conjugation and $\delta_l^{l'}$ is the Kronecker delta function. The addition formula for spherical harmonics gives 

$$  \sum_{m=-l}^l Y_{lm}(\theta,\varphi)Y_{lm}^{*}(\theta,\varphi)=\frac{2l+1}{4\pi}.$$
The Bessel function $J_\nu(\cdot)$ of the first kind of order $\nu$ is defined by 
$$J_\nu (\mu)=\sum_{n=0}^{\infty}\frac{(-1)^n}{n! \Gamma(n+\nu+1)}\bigg(\frac{\mu}{2}\bigg)^{2n+\nu},$$
where $\Gamma(\cdot)$ is the Gamma function.

It admits the following representation by the Poisson integral, see~(10.9.4) in \cite{NIST:DLMF},
$$J_\nu(\mu)=\frac{2(\mu/2)^\nu}{\sqrt{\pi}\Gamma(\nu+\frac{1}{2})}\int_{0}^{1}(1-t^2)^{\nu-\frac{1}{2}}\cos(\mu t)dt,\ \nu\geq\frac{1}{2}.$$ 
By the addition theorem for Bessel functions, see, for example, \cite{ivanov1989statistical}, p. 14, 
\begin{equation}\label{altrepres}
	\aligned
	\eta(\bold x) & = \tilde\eta(\theta,\varphi,r)
	= \pi\sqrt{2} \sum_{l=0}^{\infty} \sum_{m=-l}^l Y_{lm}(\theta,\varphi) \int_0^{\infty} \frac{J_{l+1/2}(\mu r)}{(\mu r)^{1/2}} \; Z_{lm}(d\mu) , 
	\endaligned 
\end{equation}
where $Z_{lm}(\cdot)$ is a family of random measures on $(\Bbb R_+^1, \mathcal B(\Bbb R_+^1))$, such that 
\begin{equation}\label{Zlm}
	\bold E Z_{lm}(\Delta_1) Z_{l'm'}(\Delta_2) = \delta_l^{l'} \delta_m^{m'} G(\Delta_1\cap\Delta_2),\quad \Delta_i \in \mathcal B(\Bbb R_+^1),\quad i=1,2.
\end{equation}
The stochastic integrals in (\ref{spectrepres}) and (\ref{altrepres}) are viewed as an $L_2(\Omega)$ integrals with the structural measures $F$ and $G$ correspondingly. 

Let us consider the initial conditions of the form: 
\begin{equation}\label{incondelta}
	q(\bold x,t)|_{t=0}=\delta(\bold x), \qquad \left. 
	\frac{\partial q(\bold x,t)}{\partial t}\right\vert _{t=0}=0,
\end{equation}
where $\delta(\bold x)$ is the Dirac delta-function. 

Let $Q(\bold x,t), \; \bold x\in\Bbb R^3, \; t\geq 0,$ be the fundamental solution (or the Green's function) of the initial-value problem (\ref{telegraph}) and (\ref{incondelta}) and 
\begin{equation}\label{FourierTrans}
	H( \kappa,t) = \int_{\Bbb R^3} e^{i\langle  \kappa, \bold x \rangle} \; Q(\bold x,t) \; d\bold x , 
	\quad  \kappa\in \Bbb R^3, \quad t\geq0,
\end{equation}
be its Fourier transform.
\begin{Theorem}\label{th5.11}
	The Fourier transform {\rm (\ref{FourierTrans})} {of the initial-value problem} {\rm(\ref{telegraph})} and {\rm(\ref{incondelta})} {\it is given by the formula}
	\begin{align}\label{FourierTransForm} 
		& H( \kappa,t) = \exp\left( -\frac{c^2}{2D}t \right) \\ 
		& \times \biggl\{ \biggl[ \cosh\left( ct\sqrt{\frac{c^2}{4D^2}-\Vert \kappa\Vert^2} \right)  
		+ \frac{c}{2D\sqrt{\frac{c^2}{4D^2}-\Vert \kappa\Vert^2}} \sinh\left( ct\sqrt{\frac{c^2}{4D^2}-\Vert \kappa\Vert^2} \right) \biggr] \Bbb I_{\{\Vert \kappa\Vert\le \frac{c}{2D}\}} \\ 
		& \quad + \biggl[ \cos\left( ct\sqrt{\Vert \kappa\Vert^2 - \frac{c^2}{4D^2}} \right)  
		+ \frac{c}{2D\sqrt{\Vert \kappa\Vert^2 - \frac{c^2}{4D^2}}} \sin\left( ct\sqrt{\Vert \kappa\Vert^2 - \frac{c^2}{4D^2}} \right) \biggr] \Bbb I_{\{\Vert \kappa\Vert>\frac{c}{2D}\}} \biggr\},
	\end{align}
	{\it where $\Bbb I_{\{\cdot\}}$ denotes the indicator function.}
\end{Theorem}

\begin{proof}[Proof of Theorem \ref{th5.11}]
	The Fourier transform (\ref{FourierTrans}) is the solution of the initial-value problem 
	
	\begin{equation}\label{FourierInitValue}
		\aligned 
		& \frac{1}{c^{2}}\frac{d^2H( \kappa,t)}{dt^2} + 
		\frac{1}{D}\frac{dH( \kappa,t)}{dt} + \Vert \kappa\Vert^2 H( \kappa,t) = 0 ,\\ 
		& \quad H( \kappa,t)|_{t=0} = 1, \qquad \left. \frac{\partial H( \kappa,t)}{\partial t}\right\vert _{t=0} = 0 , 
		\qquad  \kappa\in\Bbb R^3 .
		\endaligned 
	\end{equation}
	The characteristic equation for the ordinary differential equation in (\ref{FourierInitValue}) is 
	\begin{equation*}%\label{CharEq}
		\frac{1}{c^{2}} z^2 + \frac{1}{D} z + \Vert \kappa\Vert^2 = 0 ,
	\end{equation*}
	with the roots 
	\begin{equation}\label{Roots}
		z_1( \kappa) = -\frac{c^2}{2D} - \sqrt{\frac{c^4}{4D^2} - c^2\Vert \kappa\Vert^2} , \qquad 
		z_2( \kappa) = -\frac{c^2}{2D} + \sqrt{\frac{c^4}{4D^2} - c^2\Vert \kappa\Vert^2} .
	\end{equation}
	Therefore, the general solution of the ordinary differential equation in (\ref{FourierInitValue}) has the form
	\begin{equation}\label{GenSol}
		H( \kappa,t) = K_1( \kappa) e^{z_1( \kappa)t} + K_2( \kappa) e^{z_2( \kappa)t} , 
	\end{equation}
	where $K_1( \kappa), K_2( \kappa)$ are some functions that do not depend on $t$ and $z_1( \kappa), z_2( \kappa)$ are given by (\ref{Roots}). From the initial conditions in (\ref{FourierInitValue}) we obtain the system of equation to find these functions 
	\begin{align}\label{KK}
		K_1( \kappa)+K_2( \kappa)=1, \qquad z_1( \kappa)K_1( \kappa)+z_2( \kappa)K_2 ( \kappa)= 0,
	\end{align}
	which results in
	\begin{equation}\label{Coeff}
		K_1( \kappa) = \frac{1}{2} - \frac{c}{4D\sqrt{\frac{c^2}{4D^2}-\Vert \kappa\Vert^2}} ,  \qquad 
		K_2 ( \kappa)= \frac{1}{2} + \frac{c}{4D\sqrt{\frac{c^2}{4D^2}-\Vert \kappa\Vert^2}} . 
	\end{equation}
	Thus, by (\ref{GenSol}) and (\ref{Roots}) the solution of the initial-value problem (\ref{FourierInitValue}) is
	\begin{align*}
		H( \kappa,t) &= \left( \frac{1}{2} - \frac{c}{4D\sqrt{\frac{c^2}{4D^2}-\Vert \kappa\Vert^2}} \right) 
		\exp\left[ t \left( -\frac{c^2}{2D} - \sqrt{\frac{c^4}{4D^2} - c^2\Vert \kappa\Vert^2} \right) \right]\\
		& + \left( \frac{1}{2} + \frac{c}{4D\sqrt{\frac{c^2}{4D^2}-\Vert \kappa\Vert^2}} \right) 
		\exp\left[ t \left( -\frac{c^2}{2D} + \sqrt{\frac{c^4}{4D^2} - c^2\Vert \kappa\Vert^2} \right) \right]\\
		& = \exp\left( -\frac{c^2}{2D} t \right) \bigg( \frac{1}{2}\left[ \exp\left( t\sqrt{\frac{c^4}{4D^2} - c^2\Vert \kappa\Vert^2} \right) + \exp\left( -t\sqrt{\frac{c^4}{4D^2} - c^2\Vert \kappa\Vert^2} \right) \right]\\	
		& + \frac{c}{2D\sqrt{\frac{c^2}{4D^2}-\Vert \kappa\Vert^2}} \; \frac{1}{2}\left[ \exp\left( t\sqrt{\frac{c^4}{4D^2} - c^2\Vert \kappa\Vert^2} \right) - \exp\left( -t\sqrt{\frac{c^4}{4D^2} - c^2\Vert \kappa\Vert^2} \right) 
		\right]\bigg) \\
		& = \exp\left( -\frac{c^2}{2D} t \right) \biggl\{ \cosh\left( t\sqrt{\frac{c^4}{4D^2} - c^2\Vert \kappa\Vert^2}  \right) 
		+ \frac{c}{2D\sqrt{\frac{c^2}{4D^2}-\Vert \kappa\Vert^2}} \;
			\end{align*}
		\begin{align*} &\times \sinh\left( t\sqrt{\frac{c^4}{4D^2} - c^2\Vert \kappa\Vert^2}  \right) \biggr\}  = \exp\left( -\frac{c^2}{2D}t \right) 
		\biggl\{ \biggl[ \cosh\left( ct\sqrt{\frac{c^2}{4D^2}-\Vert \kappa\Vert^2} \right)\\  
		&+ \frac{c}{2D\sqrt{\frac{c^2}{4D^2}-\Vert \kappa\Vert^2}} \sinh\left( ct\sqrt{\frac{c^2}{4D^2}-\Vert \kappa\Vert^2} \right) \biggr] \Bbb I_{\{\Vert \kappa\Vert\le \frac{c}{2D}\}} 
		+ \biggl[ \cos\left( ct\sqrt{\Vert \kappa\Vert^2 - \frac{c^2}{4D^2}} \right)\\
		&+ \frac{c}{2D\sqrt{\Vert \kappa\Vert^2 - \frac{c^2}{4D^2}}} \sin\left( ct\sqrt{\Vert \kappa\Vert^2 - \frac{c^2}{4D^2}} \right) \biggr] \Bbb I_{\{\Vert \kappa\Vert>\frac{c}{2D}\}} \biggr\} .
	\end{align*}
	The theorem is proved.%\qed
\end{proof}
\begin{Remark}\label{til}
	The function $H( \kappa, t)$ given by {\rm (\ref{FourierTransForm})} is radial, that is, there exists a function $\tilde{H}(\cdot,\cdot)$ defined on $(0,\infty)\times (0,\infty)$ such that 
	$H( \kappa, t) = \tilde{H}(\Vert \kappa\Vert, t)$.
\end{Remark}
\begin{Remark}\label{remPB}
	$c/2D$ is a cut-off wave number below which the Fourier modes decay exponentially and are non-travelling as in standard heat conduction. At low wave numbers the governing PDE may be regarded as a delayed diffusion equation, as in Cattaneo's theory of heat propagation {\rm \cite{cattaneo1958forme}}. At higher wave numbers, it can easily be seen from the one-dimensional solutions that the Fourier components may be viewed as travelling waves but with exponentially decaying amplitude. At high wave numbers, the governing PDE may be regarded as a damped wave equation.	
\end{Remark}
Let us denote $\tilde H(\mu,t) = \tilde H_1(\mu,t) + \tilde H_2(\mu,t),$ such that
\begin{align}\label{ppp}
	\tilde H_1(\mu,t) & = \exp\left( -\frac{c^2}{2D}t \right)\biggl[ \cosh\left( ct\sqrt{\frac{c^2}{4D^2}-\mu^2} \right)   \notag\\ 
	& + \frac{c}{2D\sqrt{\frac{c^2}{4D^2}-\mu^2}} \sinh\left( ct\sqrt{\frac{c^2}{4D^2}-\mu^2} \right) \biggr] \Bbb I_{\{|\mu|\le \frac{c}{2D}\}} ,
\end{align}
\begin{align}
	\tilde H_2(\mu,t) & = \exp\left( -\frac{c^2}{2D}t \right)\biggl[ \cos\left( ct\sqrt{\mu^2 - \frac{c^2}{4D^2}} \right)   \notag\\ 
	& + \frac{c}{2D\sqrt{\mu^2 - \frac{c^2}{4D^2}}} \sin\left( ct\sqrt{\mu^2 - \frac{c^2}{4D^2}} \right) \biggr] \Bbb I_{\{|\mu|>\frac{c}{2D}\}} . \label{pprime}
\end{align}
\begin{Lemma}\label{lem4.2}
	It holds 
	\begin{align}\label{c4eq13}
		0 \leq \tilde H_1(\mu,t) \leq 1,
	\end{align}
	and
	\begin{equation}\label{EQ4}
		|\tilde H_2(\mu,t)|\leq \exp\bigg(-\dfrac{c^2}{2D}t\bigg)\bigg[1+\dfrac{c^2}{2D}t\bigg].
	\end{equation}
\end{Lemma}
\begin{proof}[Proof of Lemma \ref{lem4.2}]
	It follows from~(\ref{GenSol}), (\ref{KK}) and (\ref{Coeff}) that for $|\mu|\leq \frac{c}{2D}$ it holds
	\begin{align*}
		\tilde{H}_1(\mu,t)&=e^{z_2(\mu)t}(K_2(\mu)+K_1(\mu)e^{(z_1(\mu)-z_2(\mu))t})=e^{z_2(\mu)t}\bigg(1+(e^{(z_1(\mu)-z_2(\mu))t}-1)K_1(\mu)\bigg)\\
		&=e^{z_2(\mu)t}\bigg(1+(e^{-\frac{z_2(\mu)}{K_1(\mu)}t}-1)K_1(\mu)\bigg)=(1-K_1(\mu))e^{z_2(\mu)t}+K_1(\mu) e^{\big(z_2(\mu)-\frac{z_2(\mu)}{K_1(\mu)}\big)t}.
	\end{align*} 
	Note that $\tilde{H}_1(\mu,0)=1$ for $|\mu|\leq \frac{c}{2D}$ and
	\begin{align*}
		\frac{\partial \tilde{H}_1(\mu,t)}{\partial t}&=(1-K_1(\mu))z_2(\mu)e^{z_2(\mu)t}+K_1(\mu)\bigg(z_2(\mu)-\frac{z_2(\mu)}{K_1(\mu)}\bigg)e^{z_2(\mu)-\frac{z_2(\mu)}{K_1(\mu)}t}\\
		&=(1-K_1(\mu))z_2(\mu)e^{z_2(\mu)t}\bigg(1-e^{-\frac{z_2(\mu)}{K_1(\mu)}t}\bigg) \leq 0,
	\end{align*}
	because $z_2(\mu) \leq 0$ and $K_1(\mu) \leq 0$ if $|\mu | \leq \frac{c}{2D}$. Thus, $\tilde{H}_1(\mu,t) \leq \tilde{H}_1(\mu,0)=1.$

	As $\big|\frac{\sin(x)}{x}\big|\leq 1$, one obtains the upper bound~(\ref{EQ4}) from the representation~(\ref{pprime}) for $\tilde H_2(\cdot,\cdot)$.
\end{proof}
\begin{Theorem}\label{th5.2}
	The solution $q(\bold x,t) = q(\bold x,t,\omega), \; \bold x\in\Bbb R^3, \; t\geq 0, \; \omega\in\Omega,$ of the initial-value problem {\rm(\ref{telegraph})-(\ref{incond})} can be written as the convolution 
	\begin{align}\label{Convol}
		q(\bold x,t)=\int_{\Bbb R^3} e^{i( \kappa, \bold x)}H( \kappa,t) Z(d \kappa) .
	\end{align}
	The covariance function of the spatio-temporal random field {\rm (\ref{Convol})} is 
	\begin{equation}\label{Covar}
		\text{Cov}(q(\bold x,t), q(\bold x',t')) = \int_{\Bbb R^3} e^{\langle  \kappa, \bold x-\bold x' \rangle} \; 
		H( \kappa,t) \; H( \kappa,t') \; F(d \kappa) .  
	\end{equation}
\end{Theorem}
\begin{proof}[Proof of Theorem \ref{th5.2}] 
	Notice that
	\begin{equation*} 
		\aligned 
		q(\bold x,t) & = \int_{\Bbb R^3} \eta(\bold y) \; Q(\bold x - \bold y, t) \; d\bold y 
		= \int_{\Bbb R^3} \eta(\bold x - \bold z) \; Q(\bold z, t) \; d\bold z \\ 
		& = \int_{\Bbb R^3} e^{i\langle  \kappa, \bold x \rangle} \left[ \int_{\Bbb R^3} e^{i\langle  \kappa, -\bold z \rangle} Q(\bold z,t) d\bold z \right] Z(d \kappa) 
		= \int_{\Bbb R^3} e^{i\langle  \kappa, \bold x \rangle} \; H( \kappa,t) \; Z(d \kappa) ,
		\endaligned 
	\end{equation*}
	where $H( \kappa,t)$ is given by (\ref{FourierTransForm}), assuming that the random initial condition has the spectral measure $F$, such that
	\begin{equation}\label{SpectralMeasure}
		\int_{\Bbb R^3} |H( \kappa,t)|^2 \; F(d \kappa) < \infty . 
	\end{equation}
	Under the condition (\ref{SpectralMeasure}), the stochastic integral (\ref{Convol}) exists in the $L_2(\Omega)$-sense.
	
	By Lemma~\ref{lem4.2} the function $|H( \kappa,t)|$ can be bounded by a constant $C(t)$ which depends only on $t$. Noting that 
	$\int_{\Bbb R^3}|H( \kappa,t)|^2F(d \kappa) \leq C(t)B(0) $ we obtain (\ref{SpectralMeasure}). The representation~(\ref{Covar}) immediately follows from $(\ref{Convol})$ and the orthogonality of $Z(\cdot)$.
\end{proof}
\section{Spherical random hyperbolic diffusion}\label{sec5.3}
\numberwithin{equation}{section}
In this section we investigate a restriction of the spatial-temporal hyperbolic diffusion field from Section~2 to the unit sphere.

Consider the sphere $S^2=\{ \bold x\in\Bbb R^3 : \Vert\bold x\Vert = 1 \}$ in the three-dimensional Euclidean space 
$\Bbb R^3$ with the Lebesgue measure 
$$\tilde\sigma(d\bold x) = \sigma(d\theta,d\varphi) = \sin{\theta} \; d\theta \; d\varphi , \quad \theta\in [0,\pi], \; 
\varphi\in [0, 2\pi) .$$

A spatio-temporal spherical random field defined on a probability space $(\Omega, \mathcal F, P)$ is a stochastic function 
$$T(\bold x,t) = T(\bold x,t,\omega) = \tilde T(\theta,\varphi,t), \quad \bold x\in S^2, \; t \geq 0.$$ 
We consider a real-valued spatio-temporal spherical random field $T$ with zero-mean and finite second-order moments and being continuous in the mean-square sense (see, for example, Marinucci and Peccati~\cite{marinucci2011random} for definitions and other details). Under these conditions, the zero-mean random field $T$ can be expanded in the mean-square sense as the Laplace series, see~\cite{yadrenko1983spectral},
\begin{equation*}%\label{LaplaceSeries}
	\tilde T(\theta,\varphi,t) = \sum_{l=0}^{\infty} \sum_{m=-l}^l Y_{lm}(\theta,\varphi) \; a_{lm}(t) , 
\end{equation*}
where the functions $Y_{lm}(\theta,\varphi)$ represent the spherical harmonics and the coefficients $a_{lm}(t)$ are given by the formula 
\begin{equation*}%\label{Alm}
	a_{lm}(t) = \int_0^{\pi} \int_0^{2\pi} \tilde{T}(\theta,\varphi,t) \; Y_{lm}^*(\theta,\varphi) \; \sin{\theta} \; d\theta \; d\varphi .
\end{equation*}
We assume that the field $T$ is isotropic (in the weak sense), i.e. $\bold E T^2(\bold x,t) <\infty$, 
and $\bold E T(\bold x,t) T(\bold y,t') = \bold E T(g\bold x,t) T(g\bold y,t')$ for every $g\in SO(3)$, the group of rotations in $\Bbb R^3$. This is equivalent to the condition that the covariance function $\bold E \tilde T(\theta,\varphi,t) \tilde T(\theta',\varphi',t')$ depends only on the angular distance $\gamma=\gamma_{PQ}$ between the points $P=(\theta,\varphi)$ and $Q=(\theta',\varphi')$ on $S^2$ for every $t, t' \geq 0$. 

The field is isotropic if and only if 
\begin{align}\label{c43}
	\bold E a_{lm}(t) a_{l'm'}(t') = \delta_l^{l'} \delta_m^{m'} C_l(t,t') , \quad -l\le m\le l, \; -l'\le m'\le l'.
\end{align}
Hence, 
$$\bold E a_{lm}(t)a_{lm}(t') = C_l(t,t') , \quad m=0,\pm 1,\dots,\pm l.$$
The functional series $\{ C_l(t,t'),\ l=0,1,\dots \}$ is called the angular time-dependent power spectrum of the isotropic random field $\tilde T(\theta,\varphi,t)$. 

We can define a covariance function between two locations with the angular distance $\gamma$ at times $t$ and $t'$ by
\begin{equation}\label{Gamma}
	\aligned
	R(\cos\gamma, t, t')  = \bold E T(\theta,\varphi,t) T(\theta',\varphi',t')
	= \frac{1}{4\pi} \sum_{l=0}^{\infty} (2l+1) \; C_l(t,t') \; P_l(\cos\gamma) ,
	\endaligned
\end{equation}
where $P_l(x) = \frac{1}{2^l \; l!} \; \frac{d^l}{dx^l} (x^2-1)^l $
is the $l$-th Legendre polynomial. 

If $\tilde T(\theta,\varphi,t)$ is a zero-mean isotropic Gaussian field, then the coefficients $a_{lm}(t), \; m=-l,\dots,l,$ $l\ge$~$1,$ are complex-valued Gaussian stochastic processes with 
$$\bold E a_{lm}(t) = 0, \quad \bold E a_{lm}(t) a_{l'm'}(t') = \delta_l^{l'} \delta_m^{m'} C_l(t,t') .$$

By Remark~\ref{til} the random field $q(\bold x,t),\ \bold x \in \mathbb{R}^3$, given by (\ref{Convol}) is homogeneous and isotropic in $\bold x$ and, hence, its covariance function (\ref{Covar}) can be written in the form: 
$$\aligned 
\text{Cov}(q(\bold x,t),q(\bold x',t')) & = \int_0^{\infty} \frac{\sin(\mu \Vert\bold x-\bold x'\Vert)}{\mu \Vert\bold x-\bold x'\Vert} \; \tilde H(\mu,t) \; \tilde H(\mu,t') \; G(d\mu) \\
& = 2\pi^2 \sum_{l=0}^{\infty} \sum_{m=-l}^l Y_{lm}(\theta,\varphi) \; Y_{lm}^*(\theta',\varphi') \\
& \times \int_0^{\infty} \frac{J_{l+1/2}(\mu r )}{(\mu r) ^{1/2}} \; \frac{J_{l+1/2}(\mu r' )}{(\mu r') ^{1/2}} \; \tilde H(\mu,t) \; \tilde H(\mu,t') \; G(d\mu) ,
\endaligned$$
where $(r,\theta,\varphi)$ and $(r',\theta',\varphi')$ are spherical coordinates of $\bold x $ and $ \bold x'$ respectively.

Using the Karhunen theorem we obtain the following spectral representation of the random field: 
\begin{equation}\label{RandField}
	q(\bold x,t) = \tilde q(r, \theta,\varphi,t) = \pi\sqrt{2} \sum_{l=0}^{\infty} \sum_{m=-l}^l Y_{lm}(\theta,\varphi)  \int_0^{\infty} \frac{J_{l+1/2}(r\mu )}{(r\mu) ^{1/2}} \; \tilde H(\mu,t) \; Z_{lm}(d\mu) , 
\end{equation}
where the random measures $Z_{lm}(\cdot)$ are given in (\ref{Zlm}).

Similarly to the condition (\ref{SpectralMeasure}) the isotropic measure $G(\cdot)$ satisfies the following condition 
\begin{equation*}%\label{IsotropMeas}
	\int_0^{\infty} \mu^2 \; |\tilde H(\mu,t)|^2 \; G(d\mu) < \infty 
\end{equation*}
if the field has a finite variance.

Subclasses of covariance functions of the isotropic fields on the sphere can be obtained from covariance functions of homogeneous isotropic random fields in Euclidean space, since a restriction of the homogeneous and isotropic random field to the sphere yields an isotropic spherical field, see, for example, \cite{yadrenko1983spectral}, p. 76.

Consider two locations $\bold x$ and $\bold x'$ on the unit sphere $S^2$ with the angle $\gamma\in [0,\pi]$ between them. Then the Euclidean distance between these two points is $2\sin{\frac{\gamma}{2}}$, the inner product is $\langle \bold x, \bold x' \rangle = \cos{\gamma}$, which gives a direct correspondence between the covariance function $R
_0(\Vert\bold x - \bold x'\Vert,t,t')$ in the Euclidean space and the covariance function  $R(\cos{\gamma},t,t') = R_0(2\sin{\frac{\gamma}{2}},t,t')$ on the sphere for every fixed $t,t'\geq0$. Thus, the restriction of the homogeneous and isotropic hyperbolic diffusion field (\ref{RandField}) to $S^2$ is an isotropic spherical random field for every fixed $t,t'\geq 0$. We will call it the spherical hyperbolic diffusion isotropic random field $T_H(\bold x,t), \; \bold x\in S^2, \; t\geq0$. 

Its covariance function is of the form: 
\begin{equation}\label{CovFunc}
	\text{Cov}(T_H(\bold x,t),T_H(\bold x',t')) = R(\cos{\gamma},t,t') = \int_0^{\infty} \frac{\sin(2\mu \sin{\frac{\gamma}{2}})}{2\mu \sin{\frac{\gamma}{2}}} \; \tilde H(\mu,t) \; \tilde H(\mu,t') \; G(d\mu).
\end{equation}
By the addition theorem for Bessel functions, the random field $T_H(\bold x,t)=\tilde T_H(\theta,\varphi,t)$ has the following spectral representation

\begin{equation}\label{SpectrRepres1}
	\tilde T_H(\theta,\varphi,t) =  \sum_{l=0}^{\infty} \sum_{m=-l}^l Y_{lm}(\theta,\varphi) \; a_{lm}(t) , 
\end{equation}
where 
\begin{equation}\label{Alm1}
	a_{lm}(t) = \pi\sqrt{2}\int_0^{\infty} \frac{J_{l+1/2}(\mu)}{\sqrt{\mu}} \; \tilde H(\mu,t) \; Z_{lm}(d\mu) 
\end{equation}
and the random measure $Z_{lm}(\cdot)$ satisfies (\ref{Zlm}).

Thus, the angular spectrum of the isotropic spherical random field $T_H(\bold x,t)$ is given by the formula
\begin{equation}\label{AngularSpectrum}
	C_l(t,t') = 2\pi^2 \int_0^{\infty} \frac{J_{l+1/2}^2(\mu)}{\mu} \; \tilde H(\mu,t) \; 
	\tilde H(\mu,t') \; G(d\mu) .
\end{equation}
Therefore, we obtained the following result. 

\begin{Theorem}\label{th3.2}
	{\it Consider the random initial-value problem} {\rm(\ref{telegraph})}-{\rm(\ref{incond})}, {\it in which $\eta(\bold x), \; \bold x\in\Bbb R^3$, is a homogeneous isotropic random field with the isotropic spectral measure $G(\cdot)$ given by}~{\rm(\ref{Zlm})}.

	Then, the restriction of the spatio-temporal hyperbolic-diffusion random field {\rm(\ref{RandField})} {\it to the sphere $S^2$ is an isotropic spatio-temporal spherical random field with the following angular spectrum}
	$$\aligned 
	C_l(t,t') & = 2\pi^2 \biggl[ \int_0^{\frac{c}{2D}} \frac{J_{l+1/2}^2(\mu)}{\mu} \; \tilde H_1(\mu,t) \; 
	\tilde H_1(\mu,t') \; G(d\mu) \\
	& \qquad\quad + \int_{\frac{c}{2D}}^{\infty} \frac{J_{l+1/2}^2(\mu)}{\mu} \; \tilde H_2(\mu,t) \; \tilde H_2(\mu,t') \; G(d\mu) \biggr].
	\endaligned$$
	{\it The field and its covariance functions are given by} {\rm(\ref{SpectrRepres1})} {\it and} {\rm(\ref{CovFunc})} 
	{\it respectively.}
\end{Theorem} 

Notice that $T_H(\bold x,0)=\eta(\bold x),\ \bold x \in S^2$. The angular power spectrum of $\eta(\bold x)$, $\bold x \in S^2$, will be denoted by $C_l$, $l=0,1,\dots$ For spherical random fields with finite variances it holds 
\begin{align}\label{3.star}
	\sum_{l=0}^{\infty}(2l+1)C_l<\infty.
\end{align}
\begin{Lemma}\label{lem3.1}
	If {\rm(\ref{3.star})} holds true, then 
	$$\sum_{l=0}^{\infty}(2l+1)C_l(t,t')<\infty.$$
\end{Lemma}
\begin{proof}[Proof of Lemma \ref{lem3.1}]
	By Theorem~\ref{th3.2}
	\begin{align}\label{tilda1}
		\sum_{l=0}^{\infty}(2l+1)C_l(t,t')&=2\pi^2\sum_{l=0}^{\infty}(2l+1)\int_{0}^{c/2D}\frac{J_{l+\frac{1}{2}}^2(\mu)}{\mu}\tilde H_1(\mu,t)\tilde H_1(\mu,t')G(d\mu)\notag\\
		&+2\pi^2\sum_{l=0}^{\infty}(2l+1)\int_{c/2D}^{\infty}\dfrac{J_{l+\frac{1}{2}}^2(\mu)}{\mu}\tilde H_2(\mu,t)\tilde H_2(\mu,t')G(d\mu)\notag\\	
		&\leq 2\pi^2 \cdot
		\sup_{\mu< \frac{c}{2D}}\big\vert \tilde H_1(\mu,t)\tilde H_1(\mu,t')\big\vert \cdot
		\sum_{l=0}^{\infty}(2l+1)\int_{0}^{c/2D}J_{l+\frac{1}{2}}^2(\mu)G(d\mu)\notag\\
		&+2\pi^2 \cdot \sup_{\mu\geq \frac{c}{2D}}\big\vert \tilde H_2(\mu,t)\tilde H_2(\mu,t')\big\vert\cdot
		\sum_{l=0}^{\infty}(2l+1)\int_{c/2D}^{\infty}
		\frac{J_{l+\frac{1}{2}}^2(\mu)}{\mu}G(d\mu).
	\end{align}
	Now, combining (\ref{tilda1}) and Lemma~\ref{lem4.2} one gets
	
	\begin{align}\label{tilda3}
		\sum_{l=0}^{\infty}(2l+1)C_l(t,t')&\leq 2\pi^2 \int_{0}^{c/2D}\sum_{l=0}^{\infty}(2l+1)\frac{J_{l+\frac{1}{2}}^2(\mu)}{\mu}G(d\mu)+\exp\bigg(-\dfrac{c^2}{D}t\bigg) \bigg(1+\dfrac{c^2}{2D}t\bigg)^2\notag\\
		& \times 2\pi^2
		\int_{c/2D}^{\infty}\sum_{l=0}^{\infty}(2l+1)\frac{J_{l+\frac{1}{2}}^2(\mu)}{\mu}G(d\mu) \leq  \sum_{l=0}^{\infty}(2l+1)C_l,
	\end{align} 
	as $\sup_{x\geq 0}(x+1)e^{-x}= 1$, $\tilde H_1(\mu,0)=\tilde H_2(\mu,0)=1$, and $C_l(0,0)=C_l$. It completes the proof.
\end{proof}
\begin{Remark}
	It follows from Lemma~{\rm\ref{lem3.1}} and the estimate $\vert P_l(\cos\theta)\vert\leq 1$ that the solution's covariance function given by~{\rm (\ref{Gamma})} is finite if the initial condition $\eta(\bold x),\ x \in S^2$, has a finite variance.
\end{Remark}

\section{Smoothness of solutions}\label{sec5.8}
In this section, we investigate the H\"{o}lder-type continuity of the solution $\tilde{T}(\theta,\varphi,t)$ on the sphere.
We demonstrate how it depends on the decay of the angular power spectrum and provide some specifications in terms of the spectral measure $G(\cdot)$. 

First, we obtain continuity of the solution with respect to the geodesic distance on the sphere. To prove it we use the approach from Corollary~5 in~\cite{broadbridge2019random}. 
\begin{Theorem}\label{th5.4}
	Let $\tilde{T}_H(\theta,\varphi,t)$ be the solution of the initial value problem {\rm(\ref{telegraph})}-{\rm(\ref{incond})} and the random initial condition $\eta(\bold x)$, $\bold x \in S^2$, has the angular power spectrum $\{C_l,l=0,1,2,\dots   \}$ satisfying the assumption
	\begin{align}\label{sum2l2}
		\sum_{l=0}^{\infty}(2l+1)^{1+2\alpha}C_l<\infty, \ \alpha \in (0,1].
	\end{align}
	\begin{itemize}
		\item[{\rm(a)}] Then, for $t>0$ $$MSE\big( \tilde{T}_H(\theta,\varphi,t)-\tilde{T}_H(\theta',\varphi',t)   \big)\leq C \sum_{l=0}^{\infty}(2l+1)^{1+2\alpha}C_l(1-\cos \gamma)^\alpha,$$
		where $\gamma$ is the angle between directions $(\theta,\varphi)$ and $(\theta',\varphi')$.
		\item[{\rm(b)}]If the measure $G(\cdot)$ has its support in $\big[\frac{c}{2D},\infty\big)$, then
		$$MSE\big( \tilde{T}_H(\theta,\varphi,t)-\tilde{T}_H(\theta',\varphi',t)   \big)\leq C \exp \bigg(-\frac{c^2}{D}t\bigg)
		\bigg(1+\frac{c^2}{2D}t\bigg)^2
		\sum_{l=0}^{\infty}(2l+1)^{1+2\alpha}C_l(1-\cos \gamma)^\alpha.$$
	\end{itemize}
\end{Theorem}
\begin{proof}[Proof of Theorem \ref{th5.4}]
	(a) It follows from~(\ref{c43}), (\ref{Gamma}), (\ref{SpectrRepres1}) and (\ref{tilda3}) that 
	\begin{align*}
		MSE\big( \tilde{T}_H(\theta,\varphi,t)-\tilde{T}_H(\theta',\varphi',t)\big)&=2Var(\tilde{T}_H(\theta,\varphi,t))-2\text{Cov}(\tilde{T}_H(\theta,\varphi,t)\tilde{T}_H(\theta',\varphi',t))\\
		&=\frac{1}{2\pi}\sum_{l=0}^{\infty}(2l+1)C_l(t,t)(1-P_l(\cos \gamma))\\
		&\leq \frac{1}{2\pi}\sum_{l=0}^{\infty}(2l+1)C_l
		(1-P_l(\cos \gamma)).
	\end{align*}
	Applying the property of Legendre polynomials,~\cite{lang2015isotropic}, p.16, 
	$$|1-P_l(\cos \gamma)|\leq 2(1-\cos \gamma)^\alpha(l(l+1))^\alpha, \ \alpha\in (0,1],$$
	one obtains the statement (a) of the theorem.
	
	(b) It follows from the proof of~(\ref{tilda3}) that in the case of $G([0,\frac{c}{2D}])=0$ it holds
	$$C_l(t,t)\leq \exp \bigg(-\frac{c^2}{D}t\bigg) \bigg(1+\frac{c^2}{2D}t   \bigg)^2 C_l .$$
	The remaining steps are similar to the proof in~(a).
\end{proof}
\begin{Theorem}\label{th55.4}
	If the measure $G(\cdot)$ has a bounded support $[0,\delta],\ \delta>0$, then
	\begin{align}\label{mse}
		MSE\big( \tilde{T}_H(\theta,\varphi,t)-\tilde{T}_H(\theta',\varphi',t)\big)\leq C (1-\cos \gamma),\quad {\rm when} \quad \gamma \to 0+,
	\end{align}
	even for the case of $\alpha=0$ in~{\rm(\ref{sum2l2})}.
\end{Theorem}
\begin{proof}[Proof of Theorem \ref{th55.4}]
	Indeed, by~{\rm(\ref{CovFunc})} we get
	\begin{align*}
		MSE\big( \tilde{T}_H(\theta,\varphi,t)-\tilde{T}_H(\theta',\varphi',t)\big)&=2\int_0^{\infty} \bigg(1-  \frac{\sin(2\mu \sin{\frac{\gamma}{2}})}{2\mu \sin{\frac{\gamma}{2}}}   \bigg)
		\;  H^2(\mu,t) \; G(d\mu)\\
		&=2\int_0^{\delta} \bigg(1-  \frac{\sin(2\mu \sin{\frac{\gamma}{2}})}{2\mu \sin{\frac{\gamma}{2}}}   \bigg)
		\;  H^2(\mu,t) \; G(d\mu).
	\end{align*}
	For $\mu \in [0,\delta]$ it holds $2\mu \sin \frac{\gamma}{2}\to 0$, when $\gamma \to 0_+$, and therefore 
	\begin{align*}
		\bigg|1-\frac{\sin(2\mu \sin{\frac{\gamma}{2}})}{2\mu \sin{\frac{\gamma}{2}}}   \bigg|=\bigg| \sum_{k=1}^{\infty} \frac{(-1)^k}{(2k+1)!} \bigg(2\mu \sin \dfrac{\gamma}{2}\bigg)^{2k+1}    \bigg| \leq \frac{\bigg(2\mu \sin \dfrac{\gamma}{2}\bigg)^{2}}{3!}.
	\end{align*}
	Hence, 
	$$MSE\big( \tilde{T}_H(\theta,\varphi,t)-\tilde{T}_H(\theta',\varphi',t)\big) \leq C \sin^2 \frac{\gamma}{2} \int_{0}^{\delta} \mu^2 H^2(\mu,t)G(d\mu)  $$
	and~{\rm(\ref{mse})} follows from Lemma~\ref{lem4.2}.
\end{proof}
The next result gives sufficient conditions to guarantee~(\ref{sum2l2}).
\begin{Theorem}\label{th5.5}
	Suppose that $\int_{0}^{\infty}e^{\mu^2/4}G(d\mu)<\infty$. Then {\rm(\ref{sum2l2})} holds true.
\end{Theorem}
\begin{proof}[Proof of Theorem \ref{th5.5}]
	By the Poisson integral representation of the Bessel function it follows
	\begin{align*}
		\sum_{l=0}^{\infty}(2l+1)^{1+2\alpha}C_l&=2\pi^2\int_{0}^{\infty}\sum_{l=0}^{\infty}(2l+1)^{1+2\alpha}J_{l+\frac{1}{2}}^2(\mu)
		\frac{G(d\mu)}{\mu}\\
		&\leq C \int_{0}^{\infty}\sum_{l=0}^{\infty}(2l+1)^{1+2\alpha} \frac{\mu^{2l+1}}{2^{2l+1}\Gamma^2(l+1)}\frac{G(d\mu)}{\mu}    \\
		&\leq C \int_{0}^{\infty} \mu \sum_{l=0}^{\infty} \frac{(\mu^2/4)^l}{l!} \frac{(2l+1)^{1+2\alpha}}{l!}\frac{G(d\mu)}{\mu}\leq C \int_{0}^{\infty}e^{\frac{\mu^2}{4}}G(d\mu),	
	\end{align*}
	as $1+2\alpha \leq 3$.
\end{proof}
\section{Short and long memory}\label{sec5.4}
\numberwithin{equation}{section}
In this section we use the representation (\ref{CovFunc}) of covariance functions to investigate the structure of dependences of $T_H(\bold x,t)$ over time. We demonstrate that conditional on the spectral isotropic measure $G(\cdot)$ of the initial random condition $\eta(\bold x),\ \bold x\in\Bbb R^3$, the random field $T_H(\bold x,t)$ can exhibit short or long-range dependence.

The random field $T_H(\bold x,t)$ will be called short-range dependent if
\begin{equation}\label{SHORT}
	\int_0^{+\infty}|R(\cos{\gamma},t+h,t)|dh<+\infty
\end{equation}
for all $t \geq 0$ and $\gamma\in[0,\pi]$.
If the integral in (\ref{SHORT}) is divergent, the field is called long-range dependent.

Results that link behaviours of covariance functions at infinity and spectral measures at the origin are called Abelian-Tauberian theorems. A very detailed overview of such results for random fields can be found in~\cite{leonenko2013tauberian}.

First we investigate the case of $\bold x=\bold x'$ in (\ref{CovFunc}), i.e. the behaviour of $R(1,t+h,t)$.
\begin{Theorem}\label{th3.3}
	{\it For} $\bold x=\bold x'$ {\it the random field} $T_H(\bold x,t)$ {\it exhibits short-range dependence if and only if} $\mu^{-2}G(d\mu)$ {\it is integrable in a neighbourhood of zero}.
\end{Theorem}
\begin{proof}[Proof of Theorem \ref{th3.3}] It follows from (\ref{ppp}), (\ref{pprime}) and (\ref{CovFunc}) that
	\[
	\int_0^{+\infty}|R(1,t+h,t)|dh=\int_0^{+\infty}\bigg|\int_0^{c/2D}\tilde H_1(\mu,t+h)\tilde H_1(\mu,t)G(d\mu)
	\]
	\[
	+\int_{c/2D}^{+\infty}\tilde H_2(\mu,t+h)\tilde H_2(\mu,t)G(d\mu)\bigg|dh.
	\]
	Using the upper bound from (\ref{EQ4}) we get 
	\[
	\int_0^{+\infty}\bigg|\int_{c/2D}^{+\infty}\tilde H_2(\mu,t+h)\tilde H_2(\mu,t)G(d\mu)\bigg|dh\leq
	\exp\bigg(-\dfrac{c^2}{2D}t\bigg)\bigg[1+\dfrac{c^2}{2D}t\bigg]\cdot G\bigg(\big[\frac{c}{2D},+\infty\big)\bigg)
	\]
	\begin{equation}\label{EQ5}
		\times\int_0^{+\infty}\exp\bigg(-\dfrac{c^2}{2D}h\bigg)\bigg[1+\dfrac{c^2}{2D}(t+h)\bigg]dh<+\infty.
	\end{equation}
	Hence, to study the integrability of the covariance function $|R(1,t+h,t)|$ one has to investigate the integral
	\begin{equation}\label{EQ6}
		\int_0^{+\infty}\bigg|\int_0^{c/2D}\tilde H_1(\mu,t+h)\tilde H_1(\mu,t)G(d\mu)\bigg|dh.
	\end{equation}
	As $\tilde H_1(\mu,t)>0$ for $|\mu|\leq \dfrac{c}{2D}$, $t \geq 0$, it is equivalent to studying the integral
	\[
	\int_0^{c/2D}\int_0^{+\infty}\tilde H_1(\mu,t+h)\tilde H_1(\mu,t)dh\ G(d\mu),
	\]
	or, by (\ref{ppp}) and $\cosh\big(\frac{c^2t}{2D}\sqrt{1-\frac{4D^2}{c^2}\mu^2}\big) \in \big[1,\cosh \frac{c^2t}{2D} \big]$ for $\mu \in \big[0, \frac{c}{2D}\big]$, to the investigating of the finiteness of the integral
	\begin{align*}
		&\int_0^{c/2D}\int_0^{+\infty}\left(\exp\bigg(-\dfrac{c^2}{2D}h\bigg(1-\sqrt{1-\frac{4D^2}{c^2}\mu^2}\bigg)\bigg)\bigg[1+\frac{1}{\sqrt{1-\frac{4D^2}{c^2}\mu^2}}\bigg]\right.\\
		&\left.-\exp\left(-\dfrac{c^2}{2D}h\bigg(1+\sqrt{1-\frac{4D^2}{c^2}\mu^2}\bigg)\right)\dfrac{1}{\sqrt{1-\frac{4D^2}{c^2}\mu^2}}\right)\left(1+\dfrac{\sinh\bigg(\dfrac{c^2t}{2D}\sqrt{1-\frac{4D^2}{c^2}\mu^2}\bigg)}{\sqrt{1-\frac{4D^2}{c^2}\mu^2}}\right)dh\ G(d\mu)\\
		&=\dfrac{2D}{c^2}\int_0^{c/2D}\left(\dfrac{1}{1-\sqrt{1-\frac{4D^2}{c^2}\mu^2}}+\left(\dfrac{1}{1-\sqrt{1-\frac{4D^2}{c^2}\mu^2}}-\dfrac{1}{1+\sqrt{1-\frac{4D^2}{c^2}\mu^2}}\right)\dfrac{1}{\sqrt{1-\frac{4D^2}{c^2}\mu^2}}\right)
	\end{align*}	
	\begin{align*}
		&\times\left(1+\dfrac{\sinh\bigg(\dfrac{c^2t}{2D}\sqrt{1-\frac{4D^2}{c^2}\mu^2}\bigg)}{\sqrt{1-\frac{4D^2}{c^2}\mu^2}}\right)G(d\mu).
	\end{align*}
	Noting that $\dfrac{\sin(h)}{h}\in\left[0,\dfrac{\sinh(A)}{A}\right]$ on $[0,A]$, $A>0$, we obtain that (\ref{EQ6}) is finite if and only if the following integral converges
	\[
	\int_0^{c/2D} \left(\dfrac{1}{1-\sqrt{1-\frac{4D^2}{c^2}\mu^2}}+\dfrac{c^2}{{2D^2}\mu^2}\right)G(d\mu)=\dfrac{c^2}{4D^2}\int_0^{c/2D}\dfrac{3+\sqrt{1-\frac{4D^2}{c^2}\mu^2}}{\mu^2}G(d\mu).
	\]
	The last integral is finite only if $\int_{0}^{\varepsilon}\frac{G(d\mu)}{\mu^2}< \infty,\ \varepsilon>0$, which completes the proof.%\qed
\end{proof}

Now we extend Theorem~\ref{th3.3} to the case of arbitrary $\bold x$ and $\bold x'$ from $ S^2$.
\begin{Theorem}\label{th5.7}
	{\it The random field} $T_H(\bold x,t)$ {\it is short-range dependent if and only if} $\mu^{-2}G(d\mu)$ {\it is integrable in a neighbourhood of the origin}.
\end{Theorem}
\begin{proof}[Proof of Theorem \ref{th5.7}]  Note that by (\ref{CovFunc}) the integrators in $R(\cos{\gamma},t',t)$ and $R(1,t',t)$ differ only by a multiplier $\frac{\sin\left(2\mu\sin \frac{\gamma}{2}\right)}{2\mu\sin \frac{\gamma}{2}}$.
	
	Thus, 
	\[
	\int_0^{+\infty}|R(\cos{\gamma},t+h,t)|dh=\int_0^{+\infty}\bigg|\int_0^{c/2D}\dfrac{\sin\left(2\mu\sin \frac{\gamma}{2}\right)}{2\mu\sin \frac{\gamma}{2}}\tilde H_1(\mu,t+h)\tilde H_1(\mu,t)G(d\mu)
	\]
	\[
	+\int_{c/2D}^{+\infty}\dfrac{\sin\left(2\mu\sin \frac{\gamma}{2}\right)}{2\mu\sin \frac{\gamma}{2}}\tilde H_2(\mu,t+h)\tilde H_2(\mu,t)G(d\mu)\bigg|dh.
	\]
	It follows from the estimates (\ref{EQ4}), (\ref{EQ5}) and the inequality $\big\vert \frac{\sin(x)}{x}\big\vert\leq 1$ that 
	\[
	\int_{0}^{+\infty}\bigg|\int_{c/2D}^{+\infty}\dfrac{\sin\left(2\mu\sin \frac{\gamma}{2}\right)}{2\mu\sin \frac{\gamma}{2}}\tilde H_2(\mu,t+h)\tilde H_2(\mu,t)G(d\mu)\bigg|dh<+\infty.
	\]
	Now, note that for $\gamma\in (0,\pi) $ the interval $\left[ 0,c/2D\right) $ can be split into a finite number of subintervals $$\left[ 0,c/2D\right) =\bigcup_{k=1}^{K}\left[ \frac{\pi}{2\sin\frac{\gamma}{2}}(k-1),\frac{\pi}{2\sin\frac{\gamma}{2}}k \right)\bigcup\left[\frac{\pi}{2\sin\frac{\gamma}{2}}K,\frac{c}{2D}\right),$$ where $ K=\left[\frac{c\sin\frac{\gamma}{2}}{\pi D}\right] $ and $ [a] $ denotes the integer part of $a$. The ratio $ \frac{\sin\left(2\mu\sin \frac{\gamma}{2}\right)}{2\mu\sin \frac{\gamma}{2}} $ has the same sign on each of these subintervals. 
	Therefore, similar to the proof of Theorem~\ref{th3.3} we obtain the sufficient and necessary condition for the integrability of $\vert R(\cos{\gamma},t',t)\vert$
	$$ \int_0^{\frac{\pi}{2\sin\frac{\gamma}{2}}}\frac{\sin\left(2\mu\sin \frac{\gamma}{2}\right)}{2\mu\sin \frac{\gamma}{2}}\frac{G(d\mu)}{\mu^2}<\infty.$$	
	Note that by $\lim_{\mu \to 0}\frac{\sin(\mu)}{\mu}=1$ this condition is equivalent to the one in Theorem~\ref{th3.3}. This completes the proof. %\qed
\end{proof} 
\section{Approximations to solutions}\label{sec5.5}
This section introduces and studies approximate solutions of the initial value problem {\rm(\ref{telegraph})}-{\rm(\ref{incond})}.
A mean-square convergence rate to the diffusion field in terms of the angular power spectrum $C_l$ is obtained. Then several specifications in terms of the measure $G(\cdot)$ are discussed.

We define the approximation $ \tilde T_{H,L}(\theta,\varphi,t)$ of the truncation degree $L \in \mathbb{N}$ to the solution $ \tilde T_{H}(\theta,\varphi,t)$ given by~(\ref{SpectrRepres1}) as 
\begin{equation*}
	\tilde T_{H,L}(\theta,\varphi,t) =  \sum_{l=0}^{L-1}  Y_{lm}(\theta,\varphi) \; a_{lm}(t) ,\quad \theta \in [0,\pi],\ \varphi \in [0,2\pi),\ t \geq 0.
\end{equation*}
The next result provides the convergence rate of 
$\tilde T_{H,L}(\theta,\varphi,t)$ to $\tilde T_{H}(\theta,\varphi,t)$ when $L \to \infty.$
\begin{Theorem}\label{thc4.6}
	Let $\tilde T_{H}(\theta,\varphi,t)$ be the solution to the initial value problem {\rm(\ref{telegraph})}-{\rm(\ref{incond})} and $\tilde T_{H,L}(\theta,\varphi,t)$ the corresponding approximation of truncation degree $L \in \mathbb{N}$. Then, 
	\begin{align*}
		\sup_{t\geq 0}\Vert \tilde T_{H}(\theta,\varphi,t)- \tilde T_{H,L}(\theta,\varphi,t) \Vert_{L_2(\Omega\times S^2)}\leq \frac{1}{2\sqrt{\pi}}\bigg(\sum_{l=L}^{\infty}(2l+1)C_l\bigg)^{1/2}.
	\end{align*}
\end{Theorem}
\begin{proof}[Proof of Theorem \ref{thc4.6}] 
	Note that by properties of $a_{lm}(t)$ we get 
	$$\bold{E}(\tilde T_{H}(\theta,\varphi,t)- \tilde T_{H,L}(\theta,\varphi,t))=0$$
	for all $L \in \mathbb{N}$, $\theta \in [0,\pi]$, $\varphi \in [0,2\pi)$ and $t\geq 0 $.
	
	Then, by (\ref{c43}) and~(\ref{SpectrRepres1}) it follows that
	\begin{align*}
		\Vert \tilde T_{H}(\theta,\varphi,t)- \tilde T_{H,L}(\theta,\varphi,t) \Vert_{L_2(\Omega\times S^2)}&=\bigg( \sum_{l=L}^{\infty} \sum_{m=-l}^{l} Y_{lm}(\theta,\varphi) Y_{lm}^{*}(\theta,\varphi)
		\bold{E}(a_{lm}(t)a_{lm}^{*}(t))
		\bigg)^{1/2}\\
		&=\bigg( \sum_{l=L}^{\infty}\sum_{m=-l}^{l}Y_{lm}(\theta,\varphi)  Y_{lm}^{*}(\theta,\varphi) C_l(t,t)  \bigg)^{1/2}.
	\end{align*}
	Using the addition formula for spherical harmonics one gets
	\begin{align}\label{tilda4}
		\Vert \tilde T_{H}(\theta,\varphi,t)- \tilde T_{H,L}(\theta,\varphi,t) \Vert_{L_2(\Omega\times S^2)}=\frac{1}{2\sqrt{\pi}}\bigg( \sum_{l=L}^{\infty}(2l+1)C_l(t,t) \bigg)^{1/2}.
	\end{align}
	Finally, by (\ref{tilda3})  
	\begin{align*}
		\Vert \tilde T_{H}(\theta,\varphi,t)- \tilde T_{H,L}(\theta,\varphi,t) \Vert_{L_2(\Omega\times S^2)}\leq \frac{1}{2\sqrt{\pi}}\bigg(  \sum_{l=L}^{\infty}(2l+1)C_l \bigg)^{1/2}.
	\end{align*}
	$\qquad \qquad \qquad \qquad \qquad \qquad \qquad \qquad \qquad \qquad \qquad \qquad \qquad \qquad \qquad \qquad \qquad \qquad \qquad \qquad \qquad $
\end{proof}
In the general case of an arbitrary measure $G(\cdot)$, it is impossible to get a bound similar to {\rm(36)} in {\rm \cite{broadbridge2019random}}, i.e. 
\begin{align}\label{wbaden}
	\Vert  u(\theta,\varphi,t)-  u_{L}(\theta,\varphi,t) \Vert_{L_2(\Omega\times S^2)}\leq C \exp\bigg(-\frac{c^2t}{2D}\bigg)
	\bigg(  \sum_{l=L}^{\infty}(2l+1)C_l \bigg)^{1/2},
\end{align}
even for a sufficiently large $L$.
\begin{Theorem}\label{th55.121}
	For any fixed $C>0$ and $L\in N$ there exist $t>0$ and an initial random condition $\eta(\bold x),\ \bold x \in \mathbb{R}^3$, such that the norm of the approximation error $\tilde{T}_H(\theta, \varphi,t)-\tilde{T}_{H,L}(\theta, \varphi,t)$ does not satisfy~{\rm (\ref{wbaden})}.
\end{Theorem}
\begin{proof}[Proof of Theorem \ref{th55.121}] 
	Indeed, let us consider some $\varepsilon \in (0,1)$. 
	
	Then, $\sqrt{1-\frac{4D^2}{c^2}\mu^2}\geq 1-\varepsilon$ if $\mu \in I_\varepsilon:= \big[0, \sqrt{\frac{c^2}{4D^2}(1-(1-\varepsilon)^2)}\big]$. 
	
	Let the measure $G(\cdot)$ be concentrated on the interval $I_\varepsilon$.
	By {\rm(\ref{ppp})}, if $\mu \in I_\varepsilon$ then 
	$$\tilde H_1(\mu,t)\geq \exp\bigg(-\dfrac{c^2}{2D}t\bigg(1-\sqrt{1-\frac{4D^2}{c^2}\mu^2}\bigg)      \bigg) \geq \exp\bigg(-\frac{c^2}{2D}t\varepsilon\bigg).$$
	
	Hence, by {\rm(\ref{tilda4})} and Theorem~{\rm\ref{th3.2}} for any $C,\ L>0$, there exists $t,\ \varepsilon>0$, and the measure $G(\cdot)$ such that for the corresponding $\tilde T_{H}(\theta,\varphi,t)$ and $\tilde T_{H,L}(\theta,\varphi,t)$ it holds
	\begin{align*}
		\Vert \tilde T_{H}(\theta,\varphi,t)- \tilde T_{H,L}(\theta,\varphi,t) \Vert_{L_2(\Omega\times S^2)}&\geq\frac{1}{2\sqrt{\pi}} \exp\bigg(-\frac{c^2}{2D}t\varepsilon\bigg)	\bigg(  \sum_{l=L}^{\infty}(2l+1)C_l \bigg)^{1/2}\\
		&\geq C\exp\bigg(-\frac{c^2}{2D}t\bigg)	\bigg(  \sum_{l=L}^{\infty}(2l+1)C_l \bigg)^{1/2}.
	\end{align*}
	$\qquad \qquad \qquad \qquad \qquad \qquad \qquad \qquad \qquad \qquad \qquad \qquad \qquad \qquad \qquad \qquad \qquad \qquad \qquad \qquad \qquad $
\end{proof}
However, it is possible to obtain a rate of convergence that is exponential in $t$ if the measure $G(\cdot)$ has a bounded support.
\begin{Theorem}\label{th4.6}
	Let $\eta(\bold x),\ \bold x \in \mathbb{R}^3$, have the measure $G(\cdot)$ such that $G([0,\delta])=0$ for some $\delta \in (0, \frac{c}{2D})$. Then, for the solution $\tilde T_{H}(\theta,\varphi,t)$ of the initial value problem {\rm(\ref{telegraph})}-{\rm(\ref{incond})} and its approximation $\tilde T_{H,L}(\theta,\varphi,t)$ it holds 
	$$	\Vert \tilde T_{H}(\theta,\varphi,t)- \tilde T_{H,L}(\theta,\varphi,t) \Vert_{L_2(\Omega\times S^2)}\leq C \exp\big(-D\delta^2t\big)\bigg( \sum_{l=L}^{\infty} (2l+1)C_l \bigg)^{1/2}.$$
\end{Theorem}
\begin{proof}[Proof of Theorem \ref{th4.6}] 
	As $\frac{sinh(x)}{x} $ is an increasing function on $(0,\infty)$ it follows from $(\ref{ppp})$ that for $\mu \geq \delta$	
	\begin{align*}
		\tilde H_1(\mu,t) &\leq \exp\bigg(-\frac{c^2}{2D}t \bigg) \bigg( \exp\bigg(\dfrac{c^2}{2D}t\sqrt{1-\frac{4D^2}{c^2}\delta^2}      \bigg) + \exp\bigg(\dfrac{c^2}{2D}t\sqrt{1-\frac{4D^2}{c^2}\delta^2}      \bigg) \frac{1}{\sqrt{1-\frac{4D^2}{c^2}\delta^2}}   \bigg)\\
		&\leq   \exp\bigg(-\dfrac{c^2}{2D}t\bigg(1-\sqrt{1-\frac{4D^2}{c^2}\delta^2}    \bigg)  \bigg) 
		\bigg( 1+  \frac{1}{\sqrt{1-\frac{4D^2}{c^2}\delta^2}} \bigg)	\\
		&=\bigg(  1+ \frac{1}{\sqrt{1-\frac{4D^2}{c^2}\delta^2}}    \bigg)
		\exp\bigg(-\dfrac{c^2}{2D}t\times \frac{4D^2\delta^2}{c^2\bigg(1+\sqrt{1-\frac{4D^2}{c^2}\delta^2}\bigg)}
		\bigg) \\
		&\leq \bigg(1+\frac{1}{\sqrt{1-\frac{4D^2}{c^2}\delta^2}}   \bigg) \exp(-D\delta^2t).
	\end{align*}
	Notice that for $x \geq 0$ and $a \in (0,1)$ it holds $1+x \leq \frac{1}{a}\exp (xa)$.	
	
	Then, using the definition of $\tilde H_2(\mu,t)$ in (\ref{pprime}) we get for $t\geq0$
	\begin{align*}
		\tilde H_2(\mu,t)&\leq \exp\bigg(-\frac{c^2}{2D}t \bigg)\bigg(1+\frac{c^2}{2D}t\bigg) \leq \exp\bigg(-\frac{c^2}{2D}t \bigg)
		\frac{1}{\sqrt{1-\frac{4D^2}{c^2}\delta^2}}  \exp\bigg(\dfrac{c^2}{2D}t\sqrt{1-\frac{4D^2}{c^2}\delta^2}     \bigg) 	\\
		&\leq  \frac{1}{\sqrt{1-\frac{4D^2}{c^2}\delta^2}}  \exp\big(-D\delta^2t     \big).
	\end{align*}
	Hence, if $G([0,\delta])=0$ it follows from Theorem~\ref{th3.2} that 
	
	$$C_l(t,t) \leq \bigg(1+\frac{1}{\sqrt{1-\frac{4D^2}{c^2}\delta^2}}\bigg)^2  \exp\big(-2D\delta^2t     \big)C_l. $$
	Applying this bound to (\ref{tilda4}) we obtain the statement of the theorem.
\end{proof}
The next result follows from~(\ref{tilda4}) and the upper bound~(\ref{EQ4}) for $\tilde H_2(\mu,t)$.
\begin{Corollary}\label{cor4.1}
	If $G\big([0,\frac{c}{2D}]\big)=0$, then 
	$$	\Vert \tilde T_{H}(\theta,\varphi,t)- \tilde T_{H,L}(\theta,\varphi,t) \Vert_{L_2(\Omega\times S^2)}\leq \frac{1}{2\sqrt{\pi}} \bigg(1+\frac{c^2}{2D}t\bigg) \exp\bigg(-\frac{c^2}{2D}t\bigg)\bigg( \sum_{l=L}^{\infty} (2l+1)C_l \bigg)^{1/2}.$$
\end{Corollary}
\begin{Remark}
	The rates of convergence in Theorems~{\rm\ref{thc4.6},~\ref{th4.6}} and Corollary~{\rm\ref{cor4.1}} are sharp. Indeed, for $t=0$ one obtains 
	\begin{align*}
		\Vert \tilde T_{H}(\theta,\varphi,0)- \tilde T_{H,L}(\theta,\varphi,0) \Vert_{L_2(\Omega\times S^2)}&=
		\bigg( \sum_{l=L}^{\infty} 		
		\sum_{m=-l}^{l} Y_{lm}(\theta,\varphi) Y_{lm}^{*}(\theta,\varphi)C_l(0,0)
		\bigg)^{1/2}\\	
		&=\frac{1}{2\sqrt{\pi}}	
		\bigg( \sum_{l=L}^{\infty} (2l+1)C_l \bigg)^{1/2}.
	\end{align*}
\end{Remark}
The angular power spectrum $\{C_l$, $l=0,1, \dots\}$, of the initial random field $\eta(\bold x)$ is determined by the measure $G(\cdot)$. The following results provide some insight in the behaviour of $\sum_{l=L}^{\infty} (2l+1)C_l$ in terms of the spectral measure $G(\cdot)$.
\begin{Theorem}\label{th5.10}
	Let the angular power spectrum of $\eta(\bold x)$ be $\{C_l$, $l=0,1, \dots\}$.
	\begin{itemize}
		\item[{\rm(a)}] Then it holds
		\begin{align}\label{telda6}
			\sum_{l=L}^{\infty} (2l+1)C_l=2\pi^2 \int_{0}^{\infty}\mu\bigg(J_{L-\frac{1}{2}}(\mu)J_{L+\frac{1}{2}}^{'}(\mu)-J_{L+\frac{1}{2}}(\mu)J_{L-\frac{1}{2}}^{'}(\mu)\bigg)G(d\mu).
		\end{align}
		\item[{\rm(b)}] If $\int_{0}^{\infty} \mu^{1/3}G(d\mu)<\infty$, then 
		\begin{align}\label{telda7}
			\sum_{l=L}^{\infty} (2l+1)C_l\leq C \int_{0}^{\infty}\frac{\mu G(d\mu)}{(1+(L-\frac{3}{2})^{2}+\mu^2)^{1/3}},\qquad L \geq 2.
		\end{align}
		\item[{\rm(c)}] If the measure $G(\cdot)$ has a bounded support $[0,\delta]$, $\delta >0$, then 
		\begin{align}\label{telda8}
			\sum_{l=L}^{\infty} (2l+1)C_l\leq \frac{C}{\Gamma^2(L-\frac{1}{2})}\bigg(\frac{\delta}{2}\bigg)^{2L}, \quad L \geq 2.
		\end{align}
	\end{itemize} 
\end{Theorem}
\begin{proof}[Proof of Theorem \ref{th5.10}]
	(a) It follows from the representation 
	$$C_l=2\pi^2 \int_{0}^{\infty} \frac{J_{l+\frac{1}{2}}^2(\mu)}{\mu}G(d\mu)$$
	that 
	\begin{align}\label{telda5}
		\sum_{l=L}^{\infty}(2l+1)C_l=2\pi^2\int_{0}^{\infty}\sum_{l=L}^{\infty}(2l+1)J_{l+\frac{1}{2}}^2(\mu)\frac{G(d\mu)}{\mu}.
	\end{align}
	By von Lommel's formula, see (2.60) in \cite{baricz2017series},
	
	$$\sum_{n=0}^{\infty}(\nu +1+2n)J_{\nu+1+2n}^2(\mu)=\frac{\mu^2}{4}\big(J_\nu^2(\mu)-J_{\nu-1}(\mu)J_{\nu+1}(\mu)\big),$$
	where $\mu \in \mathbb{R}$ and $\nu>-1$, we obtain 
	\begin{align}\label{5.no}
		\sum_{l=L}^{\infty}(2l+1)J_{l+\frac{1}{2}}^2(\mu)&=2\sum_{n=0}^{\infty}\bigg(L+\frac{1}{2}+2n\bigg)J_{L+\frac{1}{2}+2n}^2(\mu)+2\sum_{l=L}^{\infty}\bigg(L+1+\frac{1}{2}+2n\bigg)J_{L+1+\frac{1}{2}+2n}^2(\mu)\notag\\
		&=\frac{1}{2}\mu^2\bigg(J_{L-\frac{1}{2}}^2(\mu)-J_{L-\frac{3}{2}}(\mu)J_{L+\frac{1}{2}}(\mu)+J_{L+\frac{1}{2}}^2(\mu)-J_{L-\frac{1}{2}}(\mu)J_{L+\frac{3}{2}}(\mu)\bigg)\notag\\
		&=\frac{1}{2}\mu^2\bigg(
		J_{L-\frac{1}{2}}(\mu)
		\big(J_{L-\frac{1}{2}}(\mu)-J_{L+\frac{3}{2}}(\mu)\big)+J_{L+\frac{1}{2}}(\mu) \big(J_{L+\frac{1}{2}}(\mu)-J_{L-\frac{3}{2}}(\mu)\big)\bigg)\notag\\
		&=\mu^2 \bigg( J_{L-\frac{1}{2}}(\mu) J_{L+\frac{1}{2}}^{'}(\mu)-J_{L+\frac{1}{2}}(\mu) J_{L-\frac{1}{2}}^{'}(\mu)
		\bigg).
	\end{align}
	Now, (\ref{telda6}) follows by substituting the last expression in (\ref{telda5}).
	
	(b) Using the inequality from \cite{landau2000bessel}
	$$|J_\nu(\mu)|\leq \frac{C}{(1+\nu^2+\mu^2)^{1/6}}$$
	we obtain that for $L \geq 2$
	\begin{align*}
		&\bigg| J_{L-\frac{1}{2}}(\mu)\bigg(J_{L-\frac{1}{2}}(\mu)-J_{L+\frac{3}{2}}(\mu)\bigg) +J_{L+\frac{1}{2}}(\mu)  \bigg(J_{L+\frac{1}{2}}(\mu)-J_{L-\frac{3}{2}}(\mu)\bigg) \bigg|\\ 
		&\leq 4 \frac{C}{\left(1+(L-\frac{3}{2})^2+\mu^2\right)^{1/6}}
	\end{align*}
	which after the substitution in (\ref{telda5}) gives (\ref{telda7}).
	
	(c) By the Poisson integral formula and the identity $\int_{0}^{1} (1-t^2)^ndt=\frac{\sqrt{\pi} \Gamma(n+1)}{2\Gamma(n+\frac{3}{2})}$ one obtains
	\begin{align}\label{telda9}
		\big| J_{L-\frac{3}{2}} (\mu)\big| \leq \frac{2(\mu/2)^{L-\frac{3}{2}}}{\sqrt{\pi}\Gamma(L-1)} \int_{0}^{1} (1-t^2)^{L-2} dt = \frac{(\mu/2)^{L-\frac{3}{2}}}{\Gamma(L-\frac{1}{2})} .
	\end{align}
	If $[0,\delta],\ \delta>0$, is the support of the measure $G(\cdot)$, then it follows from (\ref{telda5}), (\ref{5.no}) and (\ref{telda9}) that 
	$$\sum_{l=L}^{\infty}(2l+1)C_l \leq \frac{C}{2^{2L-3}\Gamma^{2}(L-\frac{1}{2})}\int_{0}^{\delta}\max \left(  \mu^{2(L-\frac{3}{2})+1}, \mu^{2(L+\frac{3}{2})+1} \right)G(d\mu),$$
	which completes the proof.
\end{proof}
\section{Numerical studies}\label{sec5.6}
\numberwithin{equation}{section}
This section presents numerical studies of the solution $T_H(\bold x,t)$, its angular spectrum and covariance functions over time. We also provide some numerical analysis of approximation errors. 

It is important to clarify that the numerical analysis in this paper is rather different from the one in \cite{broadbridge2019random} and requires more advanced approximation approaches. Namely, the stochastic model in \cite{broadbridge2019random} yielded the representation of the Laplace series coefficients $a_{lm}(t)=C[A_l(t)+B_l(t)]a_{lm}(0)$ for some functions $A_l(t)$ and $B_l(t)$ which can be explicitly computed in terms of elementary functions. However, for the model {\rm(\ref{telegraph})}-{\rm(\ref{incond})} there is no such simple functional relation that links $a_{lm}(t)$ and $a_{lm}(0)$. As a result, there are no explicit elementary functional relations between $C_l(t,t')$, $R(\cos \gamma, t,t')$ and $C_l(0,0)$, $R(\cos \gamma, 0,0)$ respectively. To compute spectral and covariance functions of $T_H(\bold x,t)$ at time $t>0$ one has to use formulae (\ref{CovFunc}), (\ref{Alm1}) and (\ref{AngularSpectrum}). These integral representations are given in terms of the spectral measure $G(\cdot)$ and stochastic measures $Z_{lm}(\cdot)$ of the initial random condition field $\eta(\bold x)$.

By (1.2.5) in \cite{ivanov1989statistical}, it follows from 
$$R(r)=\int_{0}^{\infty} \frac{\sin(\mu r)}{\mu r} G(d\mu)$$
that %\vspace{-.4 cm}
\begin{align}\label{funG}
	G(\mu)=\sqrt{\frac{2}{\pi}}\int_{0}^{\infty}J_{3/2}(ur)(ur)^{3/2}\frac{R(r)}{r}dr,
\end{align}
which can be used to compute (\ref{CovFunc}), (\ref{AngularSpectrum}) and simulate $Z_{lm}(\cdot)$ for computations in (\ref{Alm1}). However, obtaining a reliable approximation of the integral in (\ref{funG}) requires the estimation of the empirical covariance function $\hat{R}(r)$ on a dense grid.
Moreover, for the CMB data observed on the sphere the covariance function can be estimated only for distances that do not exceed its diameter.
We postpone the solution of these technical problems and analysis of real data to future publications.  

In the following examples we study properties of solutions and their approximations using simulated data. The case of a discrete measure $G(\cdot)$ is considered, i.e. the support of $G(\cdot)$ is a finite set $\{\mu_i,\ i=1,\dots,I\}$.
We employ real-valued stochastic measures $Z_{lm}(\cdot)$ that are concentrated on this set and satisfy the condition 
$$G(\mu_i)=\bold E\ Z_{lm}^2(\mu_i)=\sigma_i^2,\ i=1,\dots,I.$$
We assume that the random field $\eta(\bold x)$ is centered Gaussian. Hence, we can choose $Z_{lm}(\mu_i)\sim N(0, \sigma_i^2)$ that are independent for different $l$, $m$ and $i$.

In these settings formulae (\ref{CovFunc}), (\ref{Alm1}) and (\ref{AngularSpectrum}) take the following discrete forms
\begin{equation}\label{Disc}
	R(\cos \gamma, t,t')=\sum_{i=1}^{I} \frac{\sin(2\mu_i\sin(\frac{\gamma}{2}))}{2\mu_i\sin(\frac{\gamma}{2})}\tilde{H}(\mu_i,t)\tilde{H}(\mu_i,t')\sigma_i^2,
\end{equation}
$$a_{lm}(t)=\pi \sqrt{2} \sum_{i=1}^{I} \frac{J_{l+\frac{1}{2}}(\mu_i)}{\sqrt{\mu_i}}\tilde{H}(\mu_i,t)Z_{lm}(\mu_i),$$ 
\begin{align}\label{simCl}
	C_l(t,t')=2\pi^2 \sum_{i=1}^{I} \frac{J_{l+\frac{1}{2}}^2(\mu_i)}{\mu_i}\tilde{H}(\mu_i,t)\tilde{H}(\mu_i,t')\sigma_i^2, 
\end{align}
which are convenient for simulations.

This approach can also be used to approximate absolutely continuous spectral measures $G(\cdot)$ by considering a sufficiently large $I$, small $|\mu_i-\mu_{i+1}|$ and $\sigma_i^2=G\big([\mu_i,\mu_{i+1}]\big)$, $i=1,...,I.$ 
\begin{Example} \rm
	This example illustrates changes over time of the covariance function $R(\cos{\gamma},0,t)$ defined by (\ref{CovFunc}) and the power spectrum $C_{l}(t,t)$ defined by (\ref{AngularSpectrum}). To produce plots and computations we used the corresponding discrete equations (\ref{Disc}) and (\ref{simCl}) with values $\sigma_{i}=\frac{100}{i}$ by $i\in\{1,2,\dots,10\}$ and a discrete spectrum concentrated on the interval $[1,40]$. All computations and plots in this example are presented for the values $c=1$ and $D=1$ of the parameters in equation (\ref{telegraph}). 
	
	Figure \ref{fig:a} shows the covariance $R(\cos\gamma, t,t)$ at the time lags $t=0,\ t=0.1$ and $t=0.5$ as functions of the angular distance $\gamma$. To understand the effect of time and the angular distance $\gamma$ on the covariance function we provided 3D-plots (see Figure \ref{fig:b}) showing the covariance as a function of the time lag $t$. The plot in Figure \ref{fig:b} is normalized by dividing each value by $\max\limits_{\gamma\in[0,\pi]} R(\cos\gamma,0,0)$. It is obvious that the covariance decays through time and changes very little except values of $\gamma$ which are close to 0. 
	To understand the effect of the parameters $c$ and $D$ on the covariance function we also produced Figure~\ref{COV_c_D}. It illustrates changes of the covariance function $R(\cos{\gamma},t,t)$ at a specific time $t$ as functions of the angular distances $\gamma$ and the parameters $c$ or $D$. To produce this figure we used $t=0.1$. Figure \ref{fig:a1} displays $R(\cos{\gamma},0.1,0.1)$ for $D=1$ as a function of $c$ and the angular distances $\gamma$. While Figure \ref{fig:b1} displays $R(\cos{\gamma},0.1,0.1)$ for $c=1$ as a function of $D$ and the angular distances $\gamma$. The plots in Figure \ref{COV_c_D} are normalized by dividing each value by $\max\limits_{\gamma\in[0,\pi]} R(\cos\gamma,0.1,0.1)$. It is clear form Figure \ref{fig:a1} that the covariance decays through $c$ (also through $D$, see Figure \ref{fig:b1}) and changes very little except values of $\gamma$ which are close to 0. 
	Figure~\ref{fig:b1} demonstrates that the normalized covariance function exhibits decaying periodic behavior when $D$ increases.
	\begin {figure}[htb!]
		\centering
		\vspace*{-.7cm}
		\subfigure[$R(\cos \gamma,t,t)$ at the time lags $t=0,\ 0.1$, and $0.5$ and angular distances $\gamma$ for $c=D=$~$1.$]{\label{fig:a}\includegraphics[trim = 0cm 0cm 0cm 1.2cm,clip,width=0.45\linewidth,height=0.48\linewidth]{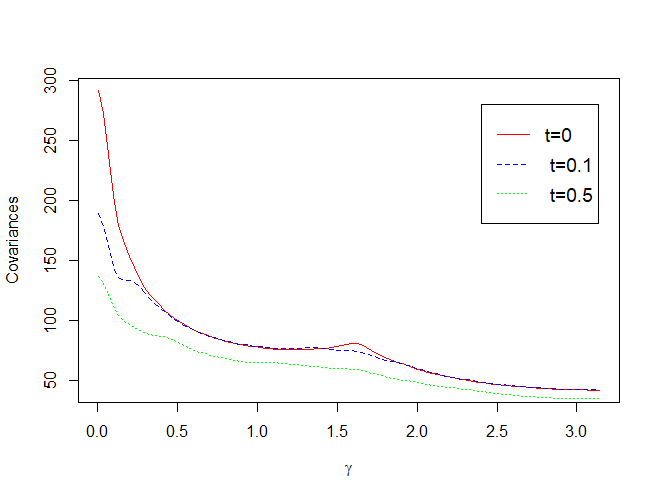}}\hspace{-0.5 cm}
		\subfigure[$R(\cos \gamma,t,t)$ for $c=D=1$ at time lag $t$ and angular distance $\gamma$. ]{\label{fig:b}	\vspace*{-.8cm}\includegraphics[trim = 4cm 2cm 5cm 1.cm,clip,width=0.55\linewidth,height=0.55\linewidth]{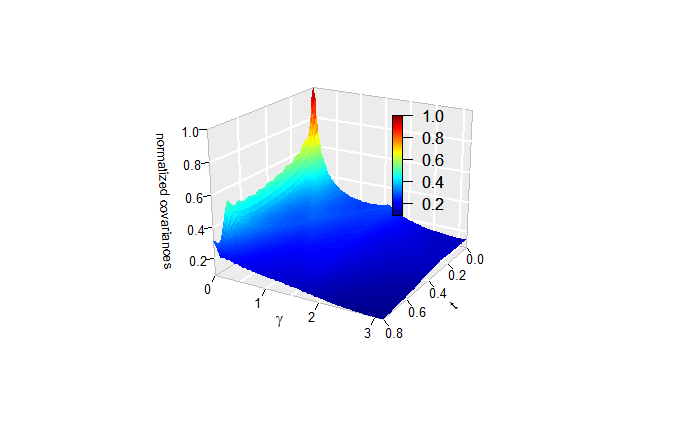}}
		\caption{}\label{fig:}
	\end{figure}
%%%%%%%%%%%%%%%%%%%%%%%%%
	\begin {figure}[h!]
		\centering
		\subfigure[$R(\cos \gamma,0.1,0.1)$ as a function of $\gamma$ and $c$ for D=1.]{\label{fig:a1}\includegraphics[trim = 3.9cm 2cm 5.5cm 1cm,clip,width=0.50\linewidth,height=0.52\linewidth]{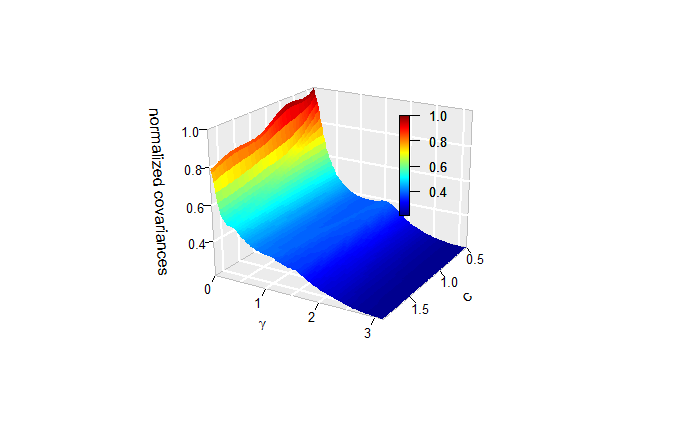}}\hspace{0.1 cm}
		\subfigure[$R(\cos \gamma,0.1,0.1)$ as a function of $\gamma$ and $D$ for $c=1$. ]{\label{fig:b1}\includegraphics[trim = 3.9cm 2cm 6cm 1cm,clip,width=0.48\linewidth,height=0.52\linewidth]{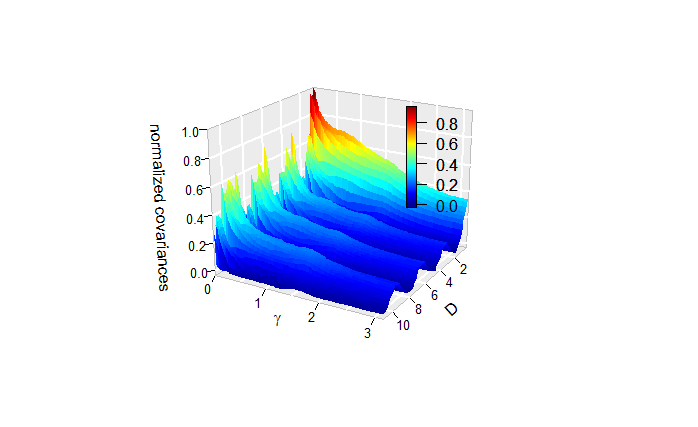}}
		\caption{}\label{COV_c_D}
	\end{figure} 
	
	Figure \ref{fig:aa} displays the power spectrum $C_{l}(t,t)$ as a function of $t\ge0$. To produce this figure we used $t\in[0,1]$ and $l=2,\ 5$, and 10. The first $70$ coefficients $C_l$ were computed by applying the equation~(\ref{simCl}) with the above values of $\sigma_{i},\ i=1,\dots,10$. From this figure it is clear that the power spectrum $C_{l}(t,t)$ decays very quick to 0 when $l$ increases. To investigate the effect of the parameter $l$ we provide a plot of the ratio $R_{0.1,0,l}=C_{l}(0.1,0.1)/C_{l}(0,0)$ for the first $70$ coefficients $C_l$ (see Figure \ref{fig:bb}). This figure confirms that the ratio $R_{0.1,0,l}$ is bounded by 1 and changes very little when $l$ increases. 
	\begin {figure}[h!]
		\centering
		\subfigure[The power spectrum $C_{l}(t,t)$ for $c=D=1$ and values $l=2,\ 5$ and 10.]{\label{fig:aa}\includegraphics[width=70mm,height=70mm]{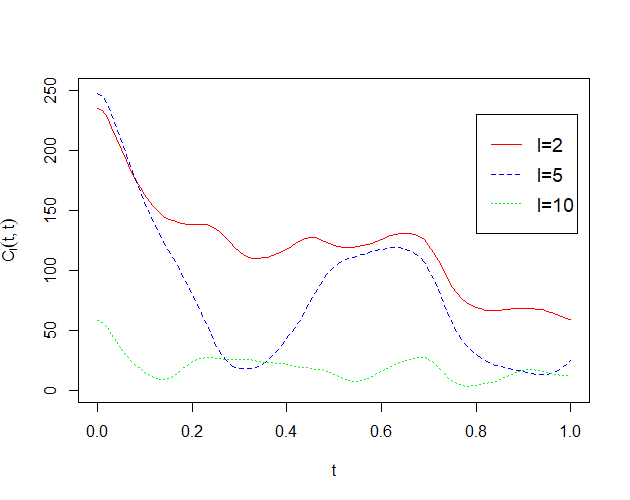}}\hspace{0.5 cm}
		\subfigure[The ratio $R_{0.1,0,l}$ of the first 70 coefficients for $c=D=$~$1$. ]{\label{fig:bb}\includegraphics[width=70mm,height=70mm]{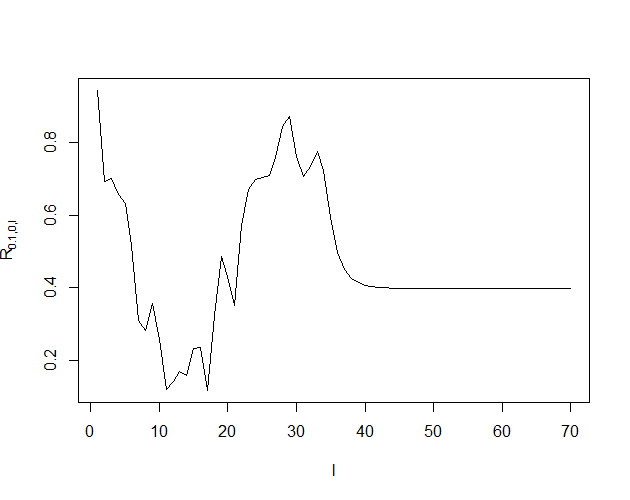}}
		\caption{}\label{figg:}
	\end{figure}
	Figure \ref{fig:c} plots the tail sums $\sum_{l\ge L}(2l+1)C_l(0,0)$ and $\sum_{l\ge L}(2l+1)C_l(0.1,0.1)$ as functions of $L$, while Figure~\ref{fig:cc} displays the corresponding ratio $RR_{0.1,0,L}=\frac{\sum_{l\ge L} (2l+1)C_{l}(0.1,0.1)}{\sum_{l\ge L}(2l+1)C_{l}(0,0)}$. From Figure~\ref{fig:c} it is clear that when $L$ increases the both terms $\sum_{l\ge L}(2l+1)C_{l}(0,0)$ and $\sum_{l\ge L}(2l+1)C_{l}(0.1,0.1)$ have the same asymptotic behaviour up to a constant multiplier which is also further confirmed in Figure~\ref{fig:cc}.
	\begin {figure}[h]
		\centering
		\subfigure[Plots of $\sum_{l\ge L}(2l+1)C_l(t,t)$ at $t=0$ and $t=0.1$.]{\label{fig:c}\includegraphics[width=70mm,height=70mm]{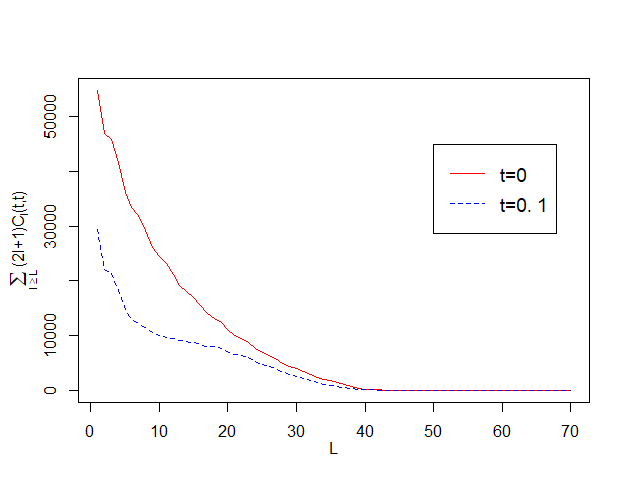}}
		\subfigure[The ratio $RR_{0.1,0,L}$. ]{\label{fig:cc}\includegraphics[width=70mm,height=70mm]{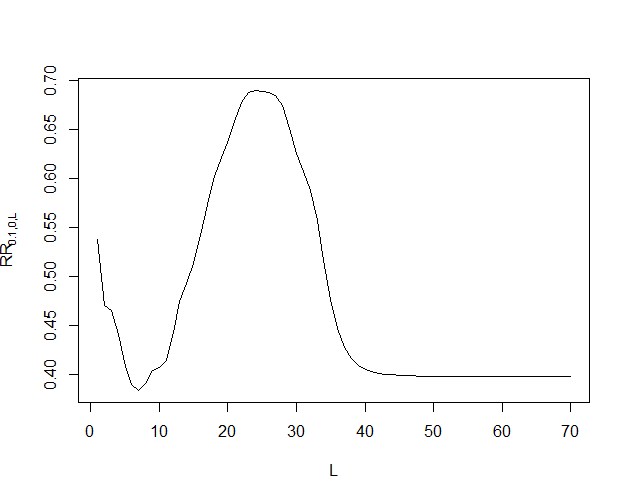}}
		\caption{}\label{figgg:}
	\end{figure}
\end{Example}
\begin{Example}
	{\rm In this example we use a discrete spectrum concentrated on the two intervals $[0,20]$ and $[80,90]$}. {\rm Thus, the initial condition random field $\eta(\bold x)$ has low and high frequency components. To produce realisations of $\eta(\bold x)$ and $T_{H}(\bold x,t), \ \bold x \in S^2,$ that are similar to small real CMB values we used $\sigma_i^2=0.00003$ and $0.0001$ for low and high frequency components respectively. These small values let us employ the visualisation tools and colour palettes used for CMB plotting in the R package rcosmo {\rm \cite{fryer2018rcosmo}} and the Python package healpy.
		
		To produce the plots and computations in this paper we use the first $100$ coefficients $C_l$ obtained by applying~(\ref{simCl}) to the above discrete spectrum. They are shown in Figure~{\rm \ref{fig5.1}} in red. 
		In this example we use the values $c=1$ and $D=2$ of the parameters in equation~(\ref{telegraph}).
		The coefficients $C_l(t,t)$ for $t=0.05$ and $0.1$ are plotted in blue and green respectively. The graph indicates two regions with relatively large values of $C_l$ that correspond to the spectral measure $G(\cdot)$ used for these computations. It can be seen that values $C_l(t,t)$ decrease over time. However, the corresponding spherical maps change rather slowly.
		Therefore, only two maps, for $t=0$ and $0.05$, are plotted in Figure~{\rm \ref{fig5.2}}.
		\begin {figure}[h]
			\vspace{-.3cm}
			\centering
			\includegraphics[width=0.6\linewidth,trim={0 0 0 1.5cm},clip]{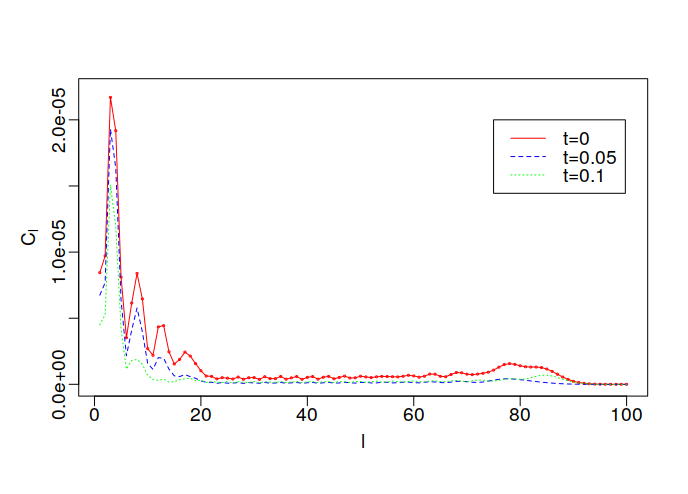} \vspace{-.5cm}
			\caption{Angular power spectra $C_l(t,t)$ for $c=1$ and $D=2$ at time $t=0,\  0.05$ and $0.1$.}\label{fig5.1}
		\end{figure}
	\begin {figure}[h]
	\centering
	\subfigure[Realisation of $T_H(\theta,\varphi,0)$ for $c=1$ and $D=2$.]{\label{fig:2a}\includegraphics[width=80mm]{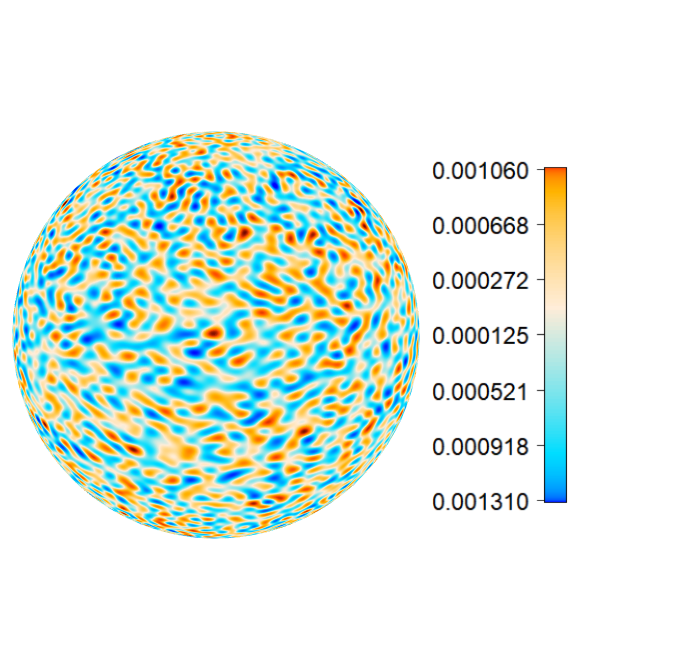}}\hspace{-8mm}
	\subfigure[Realisation of $T_H(\theta,\varphi,0.05)$ for $c=1$ and $D=2$ with two observation windows.]{\label{fig:2b}\includegraphics[width=80mm]{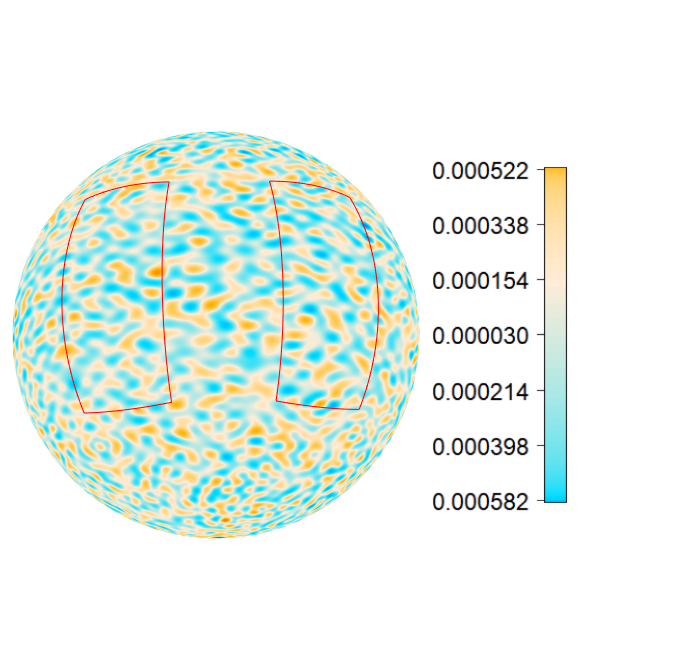}}
	\caption{}\label{fig5.2}
\end{figure}
		
		For the following numerical studies we used simulated data from two windows shown in Figure~\ref{fig:2b}. The estimated means in Table~\ref{table:5.1} confirm that $T_H(\theta,\bold x, t)$ has a zero mean. 
		It can be observed from Figure~\ref{fig5.2} and the estimated interquartile ranges (IQR) in Table~\ref{table:5.1} that the magnitude of $T_H(\bold x, t)$ values decreases with time. 
		However, the distribution type of the combined values does not change substantially. Namely, the combined values of $T_H(\bold x,t)$ exhibit approximately bell shaped behaviour with tails that are heavier than in the Gaussian case, see {\rm Figures~\ref{fig5.3} and~\ref{fig5.4}}. Similar results were obtained for various observation windows of $S^2$. For example, for the second rectangular window shown in {\rm Figure~\ref{fig:2b}} Q-Q plots and histograms of observations in this window are given in {\rm Figures~\ref{fig5.3} and~\ref{fig5.4}} respectively. These results about distributions of combined values are also confirmed by computing the Shannon entropy $$\hat{H} = - \sum_{i=1} \hat{p}_i \log (\hat{p}_i)$$ for the empirical distributions $\{\hat{p_i}\}$ given by the histograms in {\rm Figure~\ref{fig5.4}}.
		Values of $\hat{H}$ do not change much over time~$t$, see Table~\ref{table:5.1}. They are not substantially different from the entropy upper bound $\log (16)\approx 2.77$}.
	\begin{table}[h]
		\begin{center}
			\begin{tabular}{ |p{3cm}|p{2cm}|p{2cm}|p{2cm}|  }
				\hline
				Time $t$&  0 & 0.05& 10\\
				\hline
				Mean for window 1   & $1.353\cdot10^{-5}$    &$-5.62\cdot 10^{-6}$&   $3.501\cdot 10^{-7}$\\
				Mean for window 2&   $7.083\cdot 10^{-6}$  & $-1.132\cdot10^{-5}$   &$-5.166\cdot10^{-8}$\\
				IQR for window 1 &$2.877\cdot 10^{-4}$ & $1.307\cdot 10^{-4}$&  $6.78\cdot 10^{-6}$\\
				IQR for window 2    &$3.252\cdot 10^{-4}$ & $1.452\cdot 10^{-4}$&  $7.11\cdot 10^{-6}$\\
				Entropy for window 1&   $2.193$  & $2.116$&$2.369$\\
				Entropy for window 2& $2.302$  & $2.221$   &$2.387$\\
				$q$-statistics& $1.986\cdot 10^{-4}$  & $7.272\cdot 10^{-4}$&$1.5\cdot 10^{-3}$\\
				\hline
			\end{tabular}
			\caption{}
			\label{table:5.1}
		\end{center}
	\end{table}
	
	\begin {figure}[h]
		\centering
		\subfigure[Normal Q-Q plot of all $T_H(\theta,\varphi,0)$ values from window 2 in Figure~\ref{fig:2b}. ]{\label{fig:3a}\includegraphics[width=70mm]{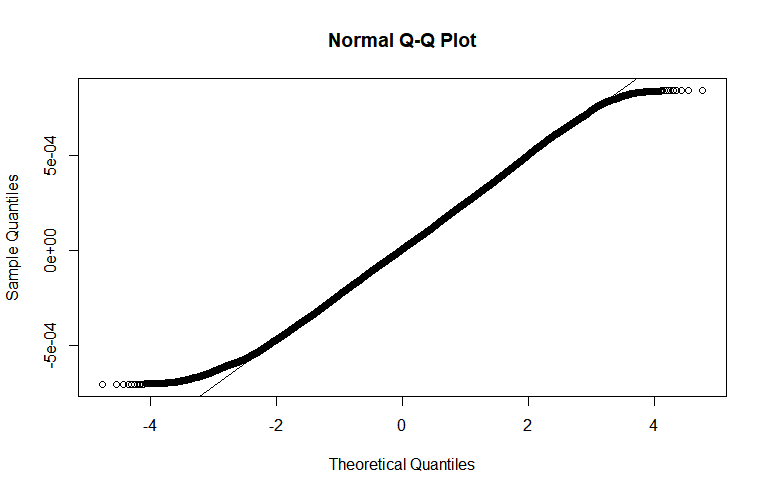}}\hspace{0.5 cm}
		\subfigure[Normal Q-Q plot of all $T_H(\theta,\varphi,0.05)$ values from window 2 in Figure~\ref{fig:2b}.]{\label{fig:3b}\includegraphics[width=70mm]{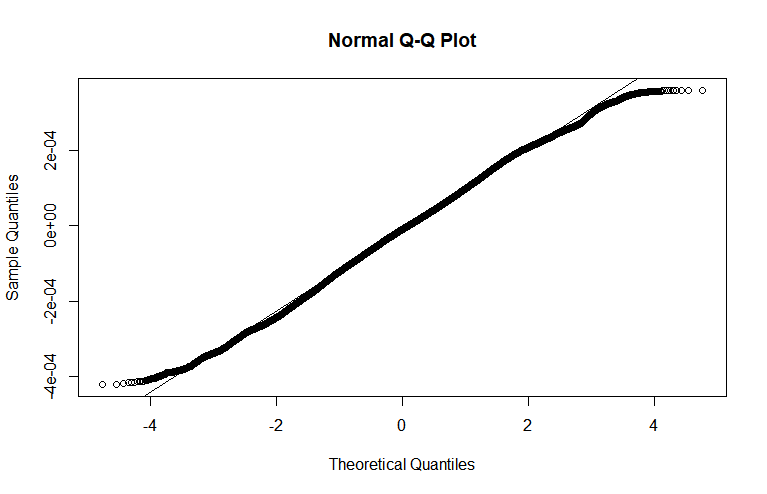}}
		\caption{}\label{fig5.3}
	\end{figure}
	
	\begin {figure}[h]
		\centering
		\subfigure[Histogram of all $T_H(\theta,\varphi,0)$ values from window 2 in Figure \ref{fig:2b}.]{\label{fig:4a}\includegraphics[width=70mm]{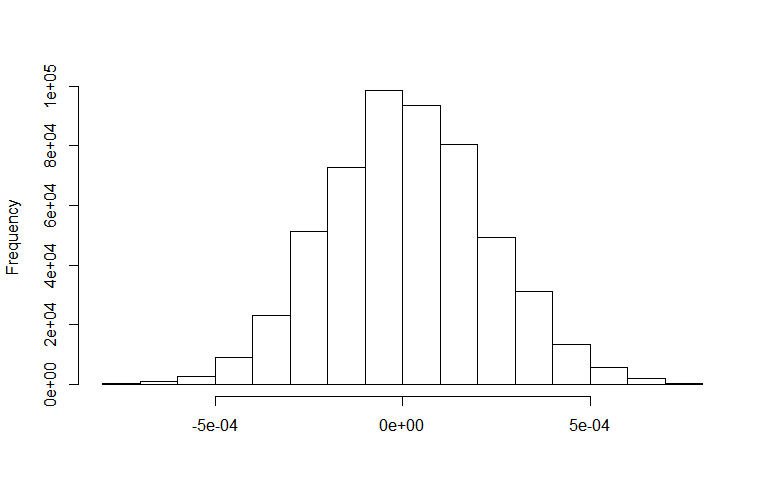}}\hspace{.5 cm}
		\subfigure[Histogram of all $T_H(\theta,\varphi,0.05)$ values from window 2 in Figure \ref{fig:2b}]{\label{fig:4b}\includegraphics[width=70mm]{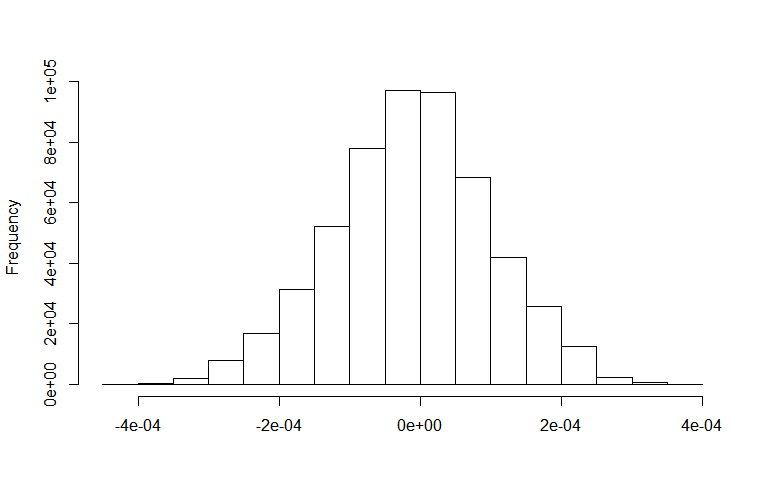}}\caption{}\label{fig5.4}
	\end{figure}
	{\rm The $q$-statistics, see~\cite{wang2016measure}, was used to investigate heterogeneity between values of $T_{H}(\theta,\varphi,t)$ in windows 1 and 2 from Figure~\ref{fig:2b}. Table~\ref{table:5.1} indicates that heterogeneity is absent at time $0$ and the evolution due to the model (\ref{telegraph}) does not introduce heterogeneity at least for short time periods. }

\end{Example}
\section{Entropy and hyperbolic diffusion}\label{sec5.9}
This section discusses the evolution of Shannon entropy for hyperbolic diffusion.
Theoretical analysis and several numeric examples are presented. To simplify the exposition and plots, only the case of $x \in \mathbb{R}$ and various non-random initial conditions are studied. 

For diffusive transport that arises from random motion of particles, the mass distribution may indeed be regarded as a probability distribution, after which Shannon entropy may be calculated. For a simple thermodynamic system governed purely by linear or nonlinear heat conduction, there is a close analogy between thermodynamic entropy and Shannon entropy (e.g. \cite {broadbridge2008entropy,Jaynes}). When the transport mechanism is modified to hyperbolic diffusion, the behaviour of entropy requires more scrutiny. In order to illustrate this, consider one dimensional solutions $q(x,t)$ on $[-\ell,\ell]\times\mathbb R^+$, subject to Neumann boundary conditions
$$q_x(x,t)=0,~~~x=\pm L.$$
This may represent transport in the $x$-direction through a linear conduit of cross section area A, with the variation of density in each cross section being effectively zero. It will be seen that the total mass M is constant. Therefore, the scaled density $q^*=qA/M$  has constant unit integral on $[-L,L]$, from which physically relevant non-negative solutions $q^*(x,t)$ may be regarded as distributions. By choosing length scale $D/c$ and time scale $D/c^2$, it may be assumed that the coefficients in the hyperbolic diffusion equation are normalised to $\pm 1$.

Let $t*=tc^2/D$, $x^*=xc/D$ and $L^*=Lc/D$. Then
$$q*_{t^*}+q^*_{t^*t^*}=q^*_{x^*x^*},$$
subject to boundary conditions
$$q^*_{x^*}=0,~~x^*=\pm L^*$$
and initial conditions
$$q^*(x^*,0)=u_0(x^*),~~q^*_{t^*}(x^*,0)=v_0(x^*).$$

Defining Shannon entropy density to be $s=-q^*\log q^*$, the hyperbolic diffusion equation for $q^*(x,t)$ implies
\begin{equation}
	s_t+\frac{D}{c^2}s_{tt}=D\frac{q^{*2}_{x}-\frac{1}{c^2}q^{*2}_t}{q^*}.
\end{equation}
The case of unbounded speed of propagation is obtained by taking the limit $c\to\infty$, which results in a positive entropy production rate $Dq^{*2}_x/q^*$. This is familiar from the theory of heat conduction, for which the entropy production rate is $L_e DT_x^2/T$, where T is absolute temperature and $L_e$ is the Lewis number, which is the order-1 ratio of thermal diffusivity to mass diffusivity.

For uni-directional waves of velocity $\pm c$, the entropy production rate is zero. For bi-directional waves, the total Shannon entropy is constant when opposite-travelling waves are not superposing, increasing when opposite-travelling superposing waves are separating, and decreasing when they are superposing and approaching. However, non-constant travelling wave solutions of the hyperbolic diffusion equation must have speed less than $c$ and they must have an amplitude that decreases with time.
For the remainder of this section, the asterisk superscripts will be conveniently omitted.\\
Some solutions of the hyperbolic diffusion equation may be of dissipative diffusive type while others may be dissipative bi-directional waves. In order to illustrate this, by the completeness of the Fourier transform, the general even solution by separation of variables is,
\begin{eqnarray}
	\label{Fourier}
	q=a_0+\sum_{n=1}^{n_c} [a_n e^{-\alpha_n^+t}+b_n e^{\alpha_n^-t}]\cos (k_nx)\\
	+\sum_{n=n_c+1}^\infty a_ne^{-0.5t}\cos (\omega_nt)\cos (k_nx),
\end{eqnarray}
where $n_c=[L/2\pi]_-$, $k_n=n\pi /L$, $\omega_n=k_n\sqrt{1-1/(2k_n)^2}$ and $\alpha_n^{\pm}=\frac 12(1\pm\sqrt{1-4k_n^2})$.

The first summation covers modes that are purely dissipative in character, just as for the linear heat diffusion equation. However in this case, the dissipative modes exist only when $L\ge 2\pi$. The second summation covers standing wave modes with decaying amplitude. These may be regarded as a superposition of a decaying left-travelling wave and a decaying right-travelling wave. Note that the dissipative mode with logarithmic decay rate $\alpha_1^-$ decays more slowly than all other modes.\\
The above solution is mass-conserving with mean value $a_0$ and constant mass integral $2La_0=1$ by normalisation. For a single decaying standing wave mode of a hyperbolic diffusion equation distribution, for some value of t,
$$q=\frac{1}{2L}[1+e^{-0.5t}\cos (\omega_{n}t)\cos(k_{n}x)].$$
Then the total Shannon entropy is
$$S=\int_{-L}^Lq\log (1/q)dx.$$
At times $t=(2m+1)\pi/2\omega_n;~~m\in\mathbb Z$, the distribution is uniform, which is the state of maximum entropy $S=\log (2L)$. Overall, the total entropy oscillates as it approaches the limiting equilibrium state. However the negative excursions of entropy may be quite small since the amplitude of oscillation decreases exponentially.

\begin{figure}
	\includegraphics[scale=0.4]{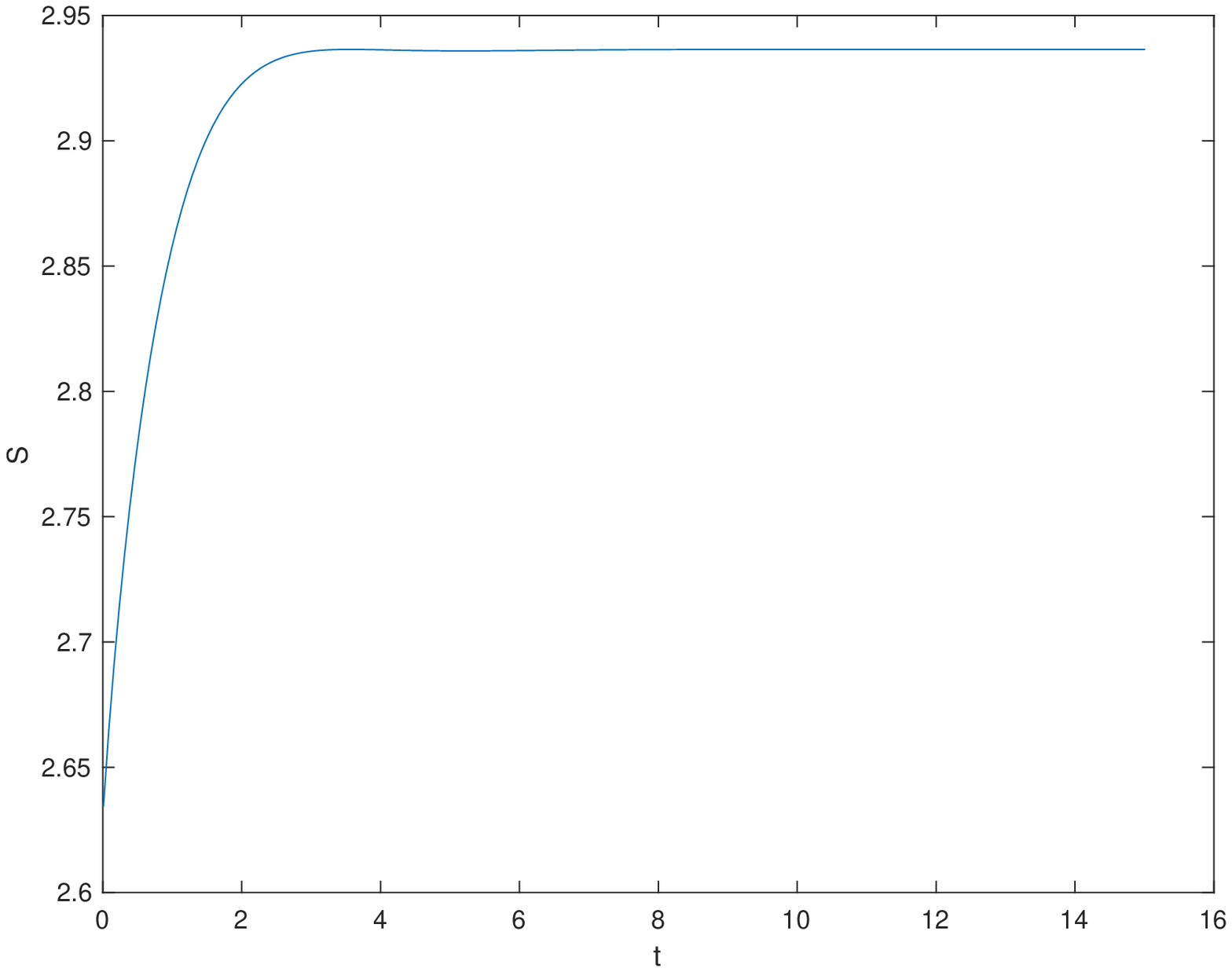}
	\nonumber
	%\centerline{\includegraphics[trim={26mm 12.5mm 10mm 0},clip,height=36mm]{eg1U0.jpg}
	%\includegraphics[trim={26mm 12.5mm 10mm 0},clip,height=36mm]{eg1U1.jpg}}
	%\hspace*{26mm}\includegraphics[trim={26mm 12.5mm 10mm 0},clip,height=36mm]{eg1U3.jpg}
	%\includegraphics[trim={26mm 12.5mm 10mm 0},clip,height=41mm]{eg1U5.jpg}
	%\caption{Nondimensional calcium ion concentration \eqref{Usymm} mapped onto a sphere for $t=0,~1,~3,~5$. The parameter values are given in the text.}
	\includegraphics[scale=0.4]{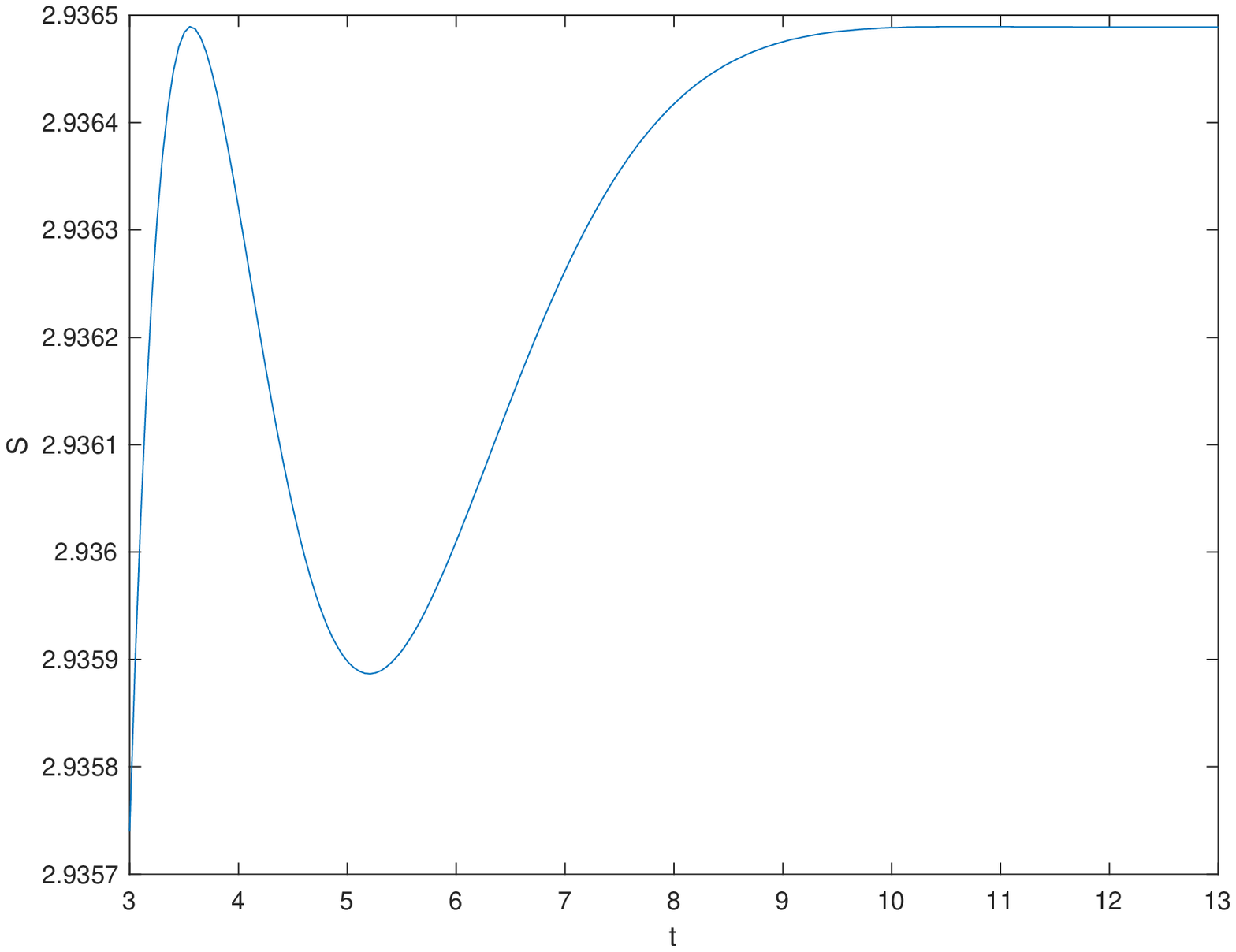}
	
	\caption{Total entropy for standing wave with single harmonic. $L=3\pi$, wave number $k_2=2\pi/L.$}	\label{EntropyS}
\end{figure} 
Figure \ref{EntropyS} plots the total entropy, calculated by trapezoidal integration with 400 intervals, versus time.\\
It would be helpful to have a point-source solution for the hyperbolic diffusion equation. As far as we are aware, there is no known simple expression for the point source evolution but it has the standard uniform Fourier spectrum that evolves according to (\ref{Fourier}). It is plotted in Figure \ref{PointSource} after truncating the Fourier series at 100 terms. As in the d'Alembert wave equation, two separating travelling delta waves emerge but now the amplitudes of the truncated spikes are decreasing and there is an additional central symmetric hump due to the purely diffusive terms. The leading edges of the spikes are travelling at maximum speed $c$. In two and three dimensions there would be similar solutions with a single travelling cylindrical or spherical shock wave surrounding a central hump.\begin{figure}
	
	\includegraphics[scale=0.4]{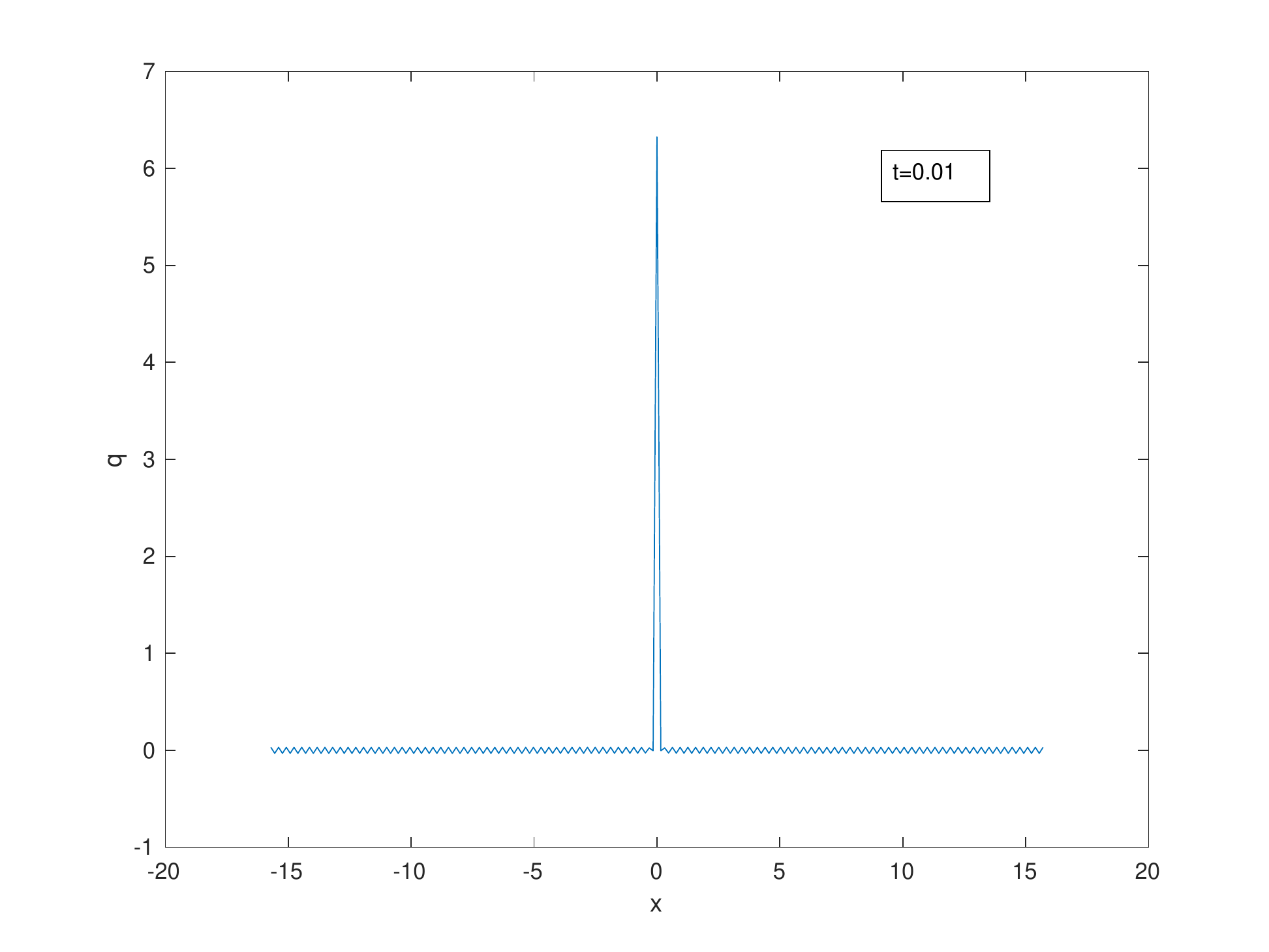}
	\includegraphics[scale=0.4]{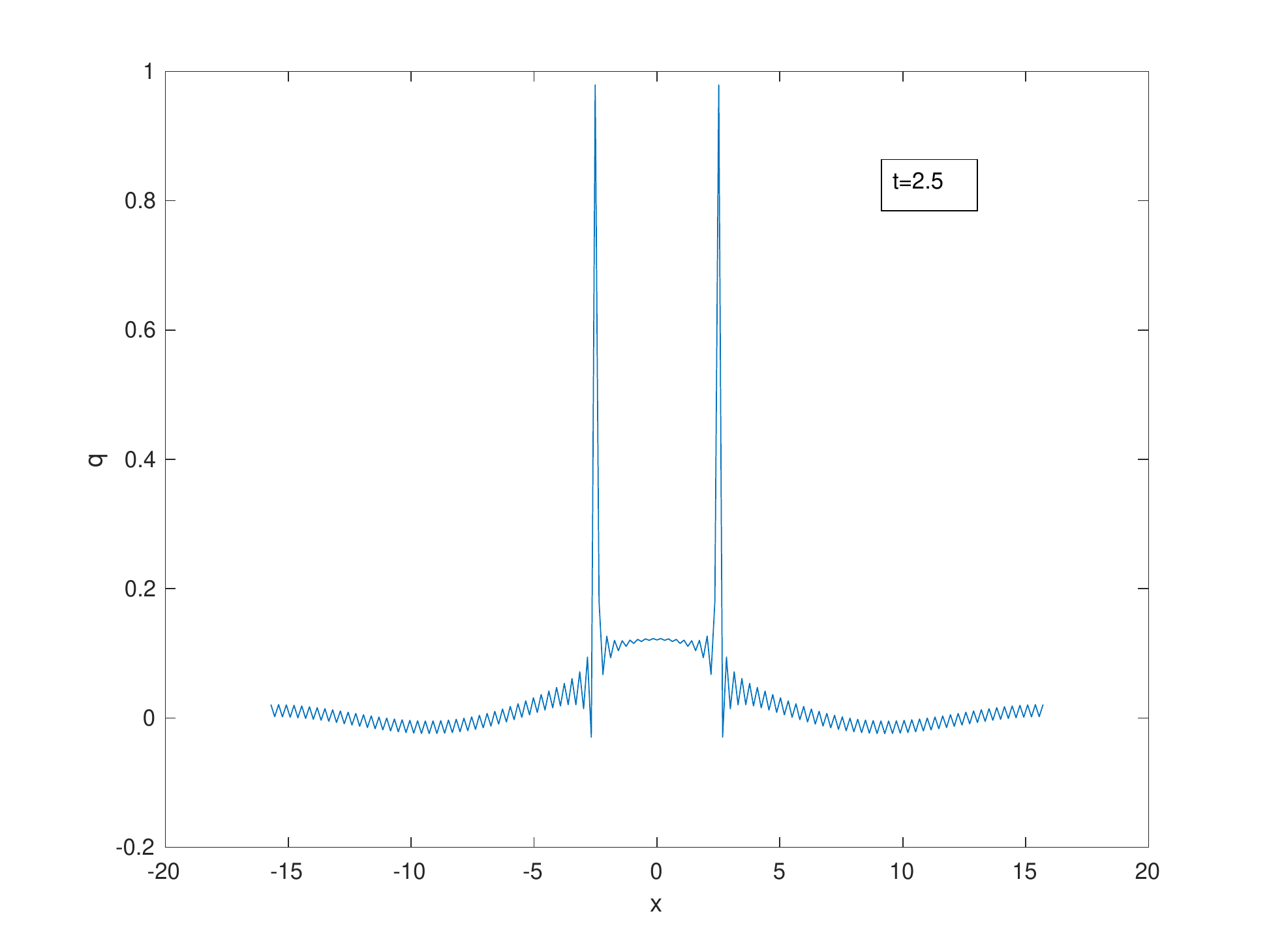}
	\caption{Evolving spike solution for $L=3\pi$.}	\label{PointSource}
\end{figure}

It is instructive also to view the motion of an initial rectangular disturbance of finite amplitude. This is approximated in Figures \ref{PointSource}-\ref{Rectangle} by a Fourier series of 200 terms. The truncated Fourier series is an exact solution but due to the truncation and the boundary conditions, the solution is negative at some values of the domain, so that Shannon entropy cannot be calculated. However, the solution is indicative of the behaviour of a non-negative solution with initial rectangle.
As in the bidirectional wave equation, the symmetric solution consists of two superposed rectangles that increase entropy as they begin to separate by travelling in opposite directions. After they have separated, their amplitude decreases which leads to further entropy increase. The height of the leading edge decreases more rapidly than the trailing edge so each rectangle evolves to a trapezoid. The leading edge, which is the boundary of the disturbance, continues to move at maximum speed $c$. Between the trapezoids, there is a central hump that eventually dominates, and resembles a diffusive Gaussian, increasing entropy further. With this kind of peaked initial condition, there is no indication of any significant period of entropy decrease.
\begin{figure}
	\includegraphics[scale=0.4]{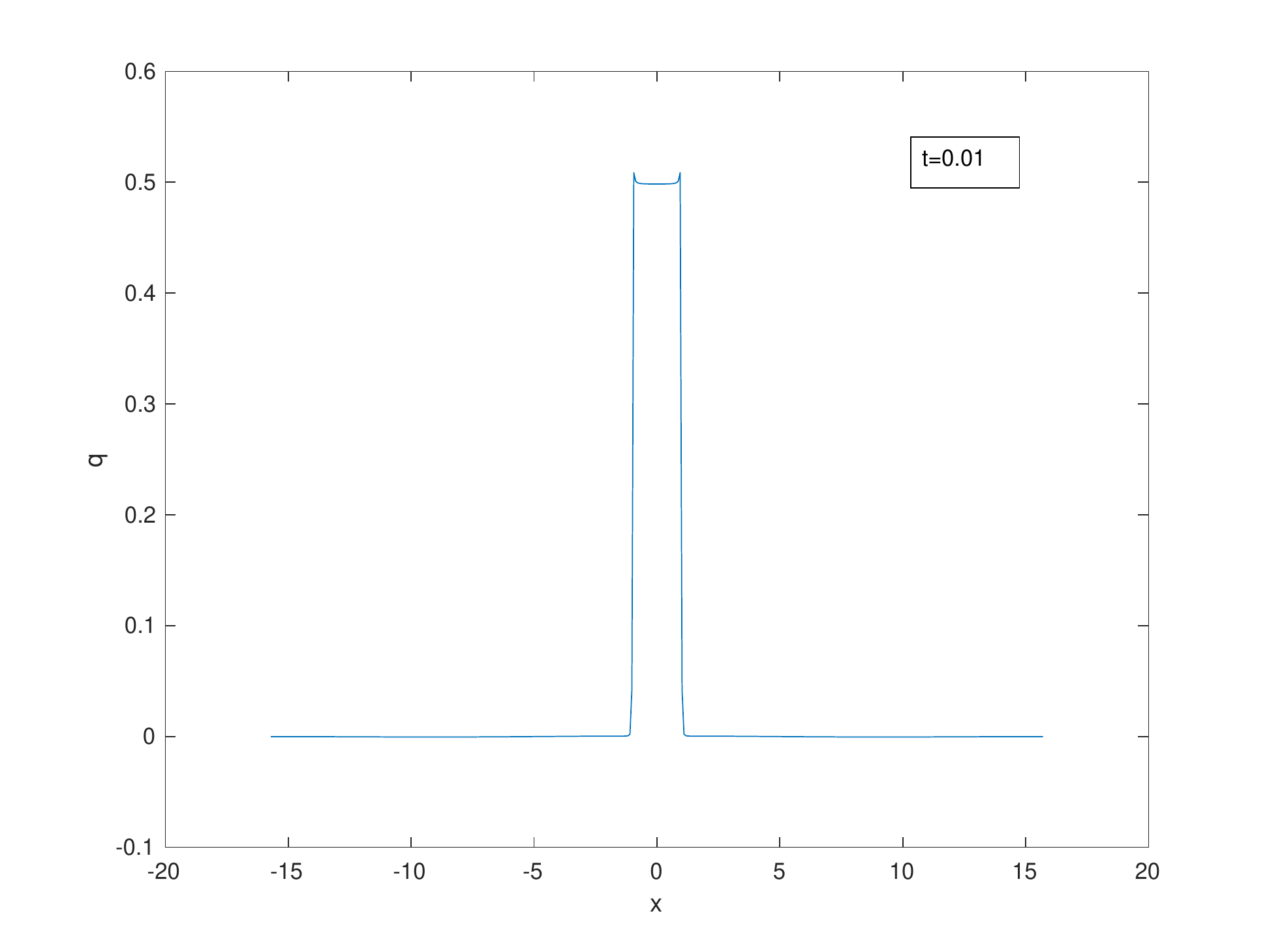}
	\includegraphics[scale=0.4]{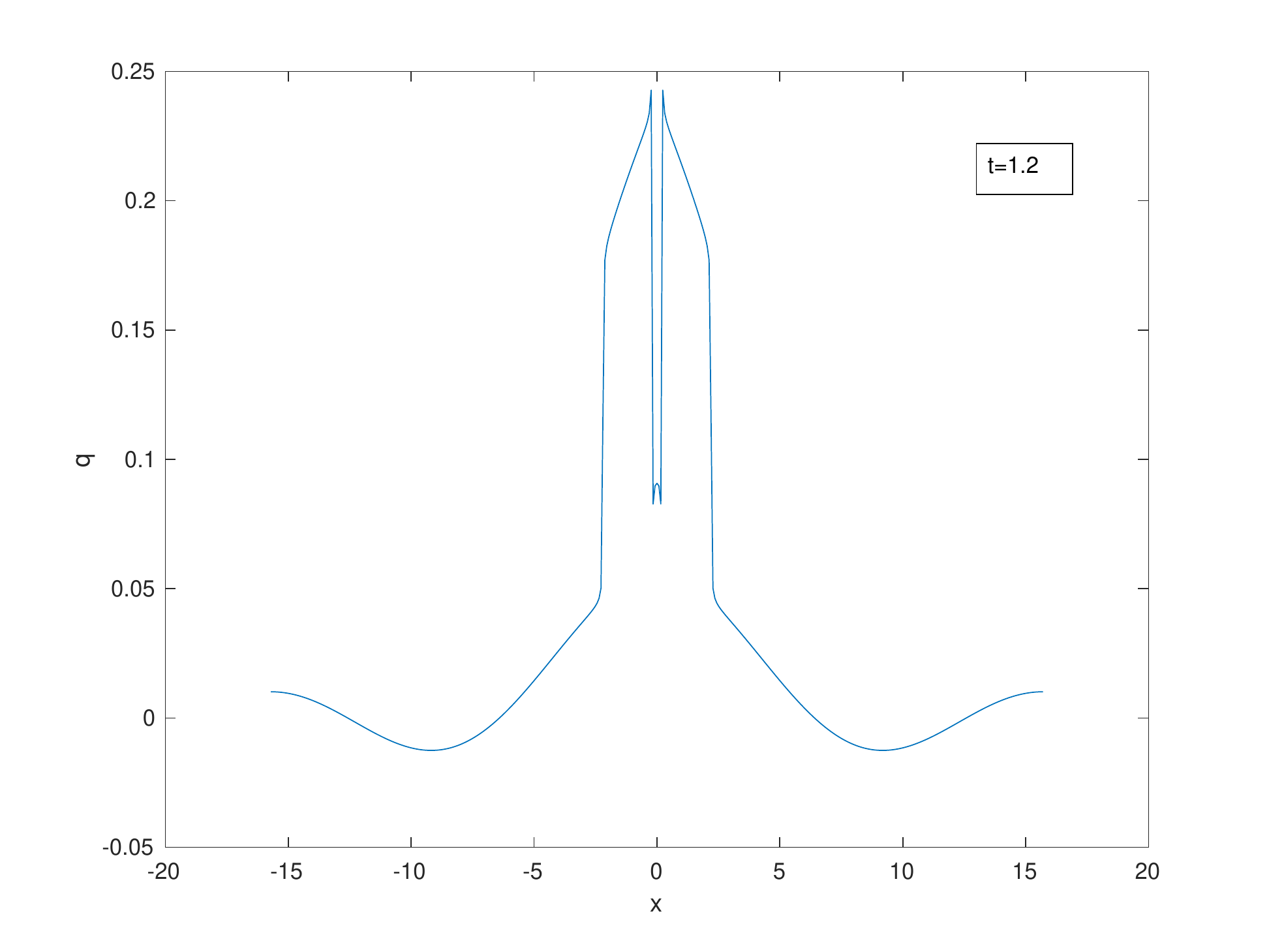}
	\caption{Evolving symmetric rectangle: emergent bidirectional wave.}	\label{Rectangle}
\end{figure}
\begin{figure}
	\includegraphics[scale=0.4]{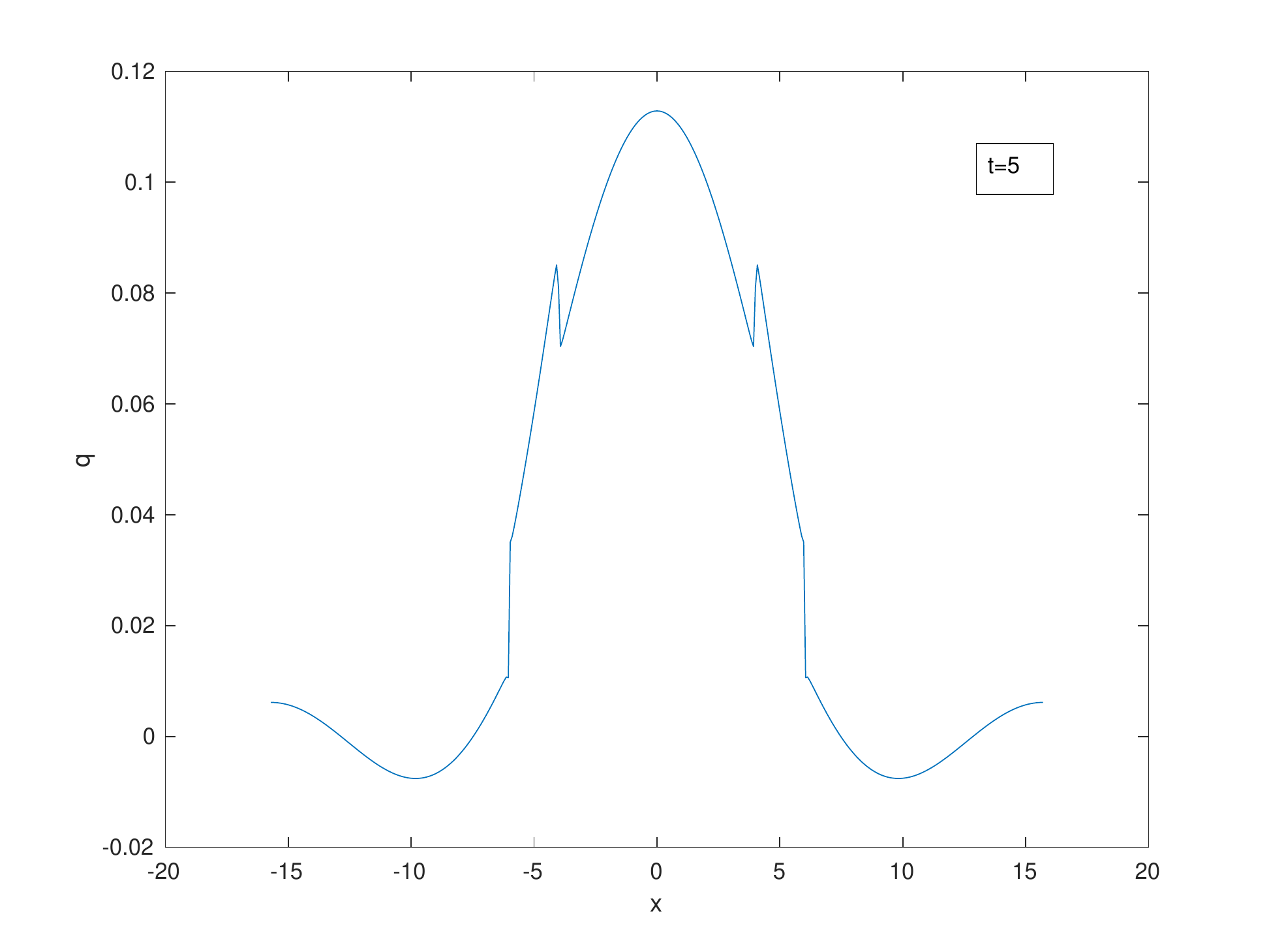}
	\includegraphics[scale=0.4]{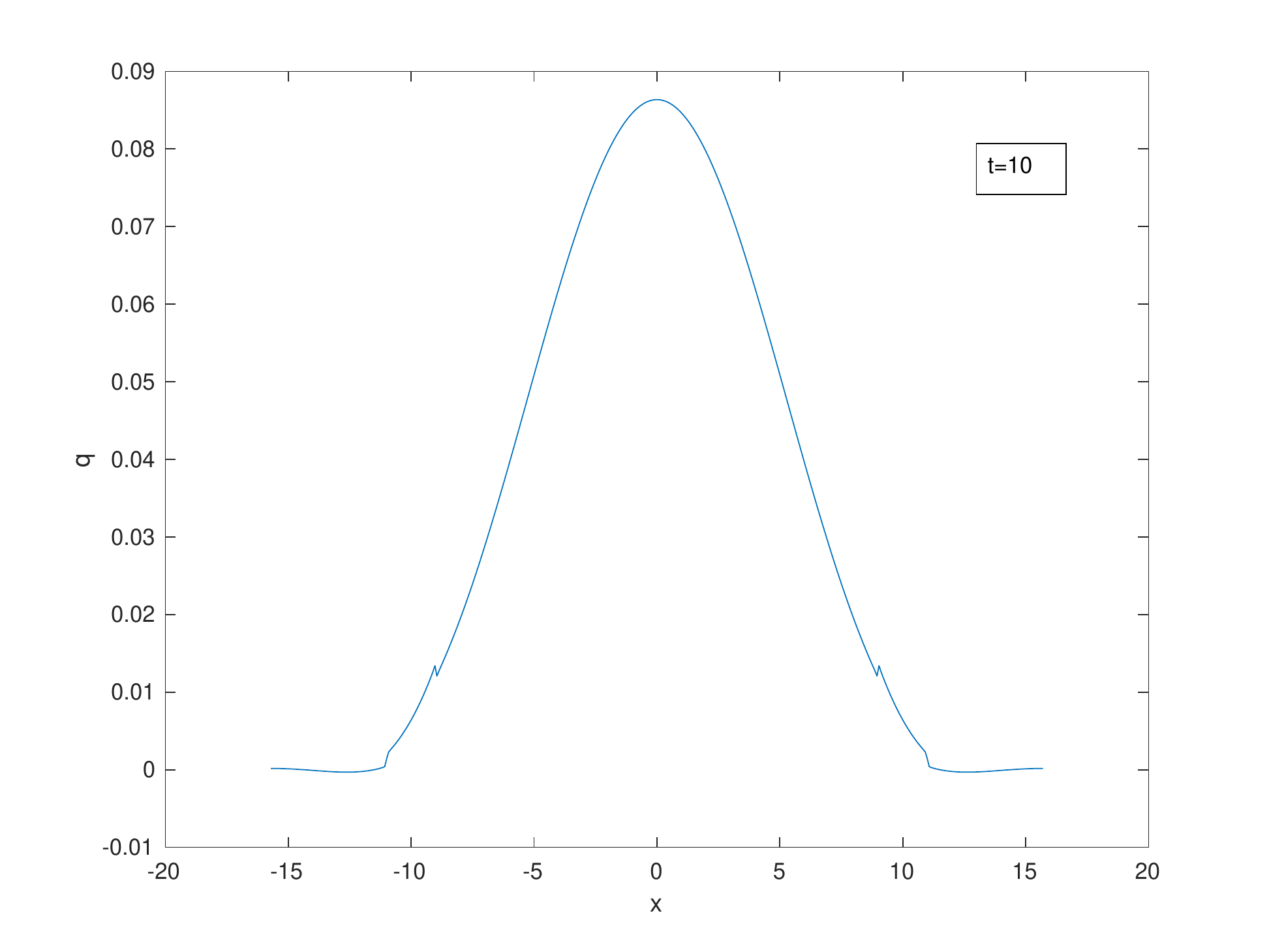}
	\caption{Evolving rectangle: dominant diffusive hump at large $t$, with leading edge of remnant rectangle demarcating the extent of the disturbance.}
\end{figure}
\section{Future research problems}
This paper investigates evolutions of random fields determined by hyperbolic diffusion equations with random initial conditions.
Spherical random fields are modeled as restrictions of 3D solutions fields to the sphere. Compared to the previous publications, it results in more realistic physical models, but arises more complicated representations that involve spectral measures of initial random conditions.
Detailed studies of the solutions and their approximations are presented. 

Some important problems and extensions for future research are:
\begin{itemize}
	\item investigating the sharpness of the obtained upper bounds on approximation errors, see \cite{broadbridge2019random};
	\item  developing statistical estimators of the equation parameters and studying their asymptotic properties;
	\item extending the methodology to tangent spherical vector fields, see~\cite{li2019fast};
	\item developing numerical methods for the obtained representations to deal with spectra of initial conditions;
	\item extending the analysis and numerical studies in Section~\ref{sec5.9} to other scenarios;
	\item in line with the theme of this special issue, in future we intend to study the effect of nonlinear diffusivity in the equation
	$$q_t+\frac{1}{c^2}q_{tt}=\nabla \cdot [D(q)\nabla q].$$ For example, if $q$ is the electron density in a plasma, $D(q)$ is typically decreasing \cite{PBJG}.
\end{itemize} 
%%%%%%%%%%%%%%%%%%%
%%%%%%%%%%%%%%%%%%%

%\begin {figure}[h]
%\centering
%\includegraphics[width=2 cm]{Definitions/logo-mdpi}
%\caption{This is a figure, Schemes follow the same formatting. If there are multiple panels, they should be listed as: (\textbf{a}) Description of what is contained in the first panel. (\textbf{b}) Description of what is contained in the second panel. Figures should be placed in the main text near to the first time they are cited. A caption on a single line should be centered.}
%\end{figure}   

%\begin{listing}[H]
%\caption{Title of the listing}
%\rule{\textwidth}{1pt}
%\raggedright Text of the listing. In font size footnotesize, small, or normalsize. Preferred format: left aligned and single spaced. Preferred border format: top border line and bottom border line.
%\rule{\textwidth}{1pt}
%\end{listing}

%% If the documentclass option "submit" is chosen, please insert a blank line before and after any math environment (equation and eqnarray environments). This ensures correct linenumbering. The blank line should be removed when the documentclass option is changed to "accept" because the text following an equation should not be a new paragraph. 

%% Example of a proof:

%%%%%%%%%%%%%%%%%%%%%%%%%%%%%%%%%%%%%%%%%%
\vspace{6pt} 

%%%%%%%%%%%%%%%%%%%%%%%%%%%%%%%%%%%%%%%%%%
%% optional
%\supplementary{The following are available online at \linksupplementary{s1}, Figure S1: title, Table S1: title, Video S1: title.}


%%%%%%%%%%%%%%%%%%%%%%%%%%%%%%%%%%%%%%%%%%

%%%%%%%%%%%%%%%%%%%%%%%%%%%%%%%%%%%%%%%%%%
%\funding{Please add: ``This research received no external funding'' or ``This research was funded by NAME OF FUNDER grant number XXX.'' and  and ``The APC was funded by XXX''. Check carefully that the details given are accurate and use the standard spelling of funding agency names at \url{https://search.crossref.org/funding}, any errors may affect your future funding.}

%%%%%%%%%%%%%%%%%%%%%%%%%%%%%%%%%%%%%%%%%%
{\bf Acknowledgments} {This research was supported under the Australian Research Council's Discovery Project DP160101366. We are also grateful for the use of data of the
	Planck/ESA mission from the Planck Legacy Archive. }

\end{document}